\documentclass[sigconf, natbib = true]{acmart}

\setcopyright{rightsretained}

\usepackage[T1]{fontenc}
\usepackage{lmodern}

\usepackage{environ}
\NewEnviron{eq}{%
\begin{equation}\begin{split}
  \BODY
\end{split}\end{equation}
}

\usepackage{bbm, enumerate, graphicx, mathtools, tikz, enumerate,color, subcaption, hyperref}

\makeatletter

\newcommand{\bld}[1]{\boldsymbol{#1}}

\newcommand{\full}{\mathrm{full}}
\newcommand{\idle}{\mathrm{idle}}
\newcommand{\qq}{\mathbf{q}}
\newcommand{\dd}{\boldsymbol{\delta}}
\newcommand{\QQ}{\mathbf{Q}}

\newcommand{\ZZ}{\mathbf{Z}}
\newcommand{\vv}{\mathbf{v}}

\newcommand{\cX}{\mathcal{X}}
\newcommand\Pro[1]{\mathbbm{P}\left(#1\right)}  
\newcommand\E[1]{\mathbbm{E}\left(#1\right)}  
\newcommand{\expt}{\mathbbm{E}}  
\newcommand{\pto}{\ensuremath{\xrightarrow{\mathbbm{P}}}}  
\newcommand{\dto}{\ensuremath{\xrightarrow{d}}}  
\newcommand{\dif}{\mathrm{d}} 

\newcommand\ind[1]{\ensuremath{\mathbbm{1}_{\left[#1\right]}}} 
\newcommand\indn[1]{\ensuremath{\mathbbm{1}_{#1}}} 
\newcommand{\N}{\mathbbm{N}}                 
\newcommand{\Z}{\mathbbm{Z}}				
\newcommand{\expn}{\mathrm{Exp}}
\newcommand{\e}{\mathrm{e}}

\newtheorem{claim}{Claim}
\newtheorem{fact}{Fact}

\let\plainqed\qedsymbol
\newcommand{\claimqed}{$\lrcorner$}
\newenvironment{claimproof}{\begin{proof}\renewcommand{\qedsymbol}{\claimqed}}{\end{proof}\renewcommand{\qedsymbol}{\plainqed}}

\makeatother

\begin{document}
\title{Optimal Service Elasticity in Large-Scale Distributed Systems}
\author{Debankur Mukherjee}
\orcid{0000-0003-1678-4893}
\affiliation{%
  \institution{Eindhoven University of Technology}
  \postcode{5600 MB}
  \city{Eindhoven} 
  \country{The Netherlands} 
}
\email{d.mukherjee@tue.nl}

\author{Souvik Dhara}
\affiliation{%
  \institution{Eindhoven University of Technology}
  \postcode{5600 MB}
  \city{Eindhoven} 
  \country{The Netherlands} 
}
\email{s.dhara@tue.nl}

\author{Sem Borst}
\affiliation{%
  \institution{Eindhoven University of Technology}
  \postcode{5600 MB}
  \city{Eindhoven} 
  \country{The Netherlands} 
}
\additionalaffiliation{%
\institution{Nokia Bell Labs, Murray Hill, NJ, USA}}
\email{s.c.borst@tue.nl}

\author{Johan S.H.~van Leeuwaarden}
\affiliation{%
  \institution{Eindhoven University of Technology}
  \postcode{5600 MB}
  \city{Eindhoven} 
  \country{The Netherlands} 
}
\email{j.s.h.v.leeuwaarden@tue.nl}

\renewcommand{\shortauthors}{Mukherjee et al.}

\begin{abstract}
A fundamental challenge in large-scale cloud networks and data
centers is to achieve highly efficient server utilization and limit energy consumption,
while providing excellent user-perceived performance
in the presence of uncertain and time-varying demand patterns.
Auto-scaling provides a popular paradigm for automatically
adjusting service capacity in response to demand while meeting
performance targets, and queue-driven auto-scaling techniques have
been widely investigated in the literature.
In typical data center architectures and cloud environments however,
no centralized queue is maintained, and load balancing algorithms
immediately distribute incoming tasks among parallel queues.
In these distributed settings with vast numbers of servers, centralized queue-driven auto-scaling
techniques involve a substantial communication overhead and major
implementation burden, or may not even be viable at all.

Motivated by the above issues, we propose a joint auto-scaling
and load balancing scheme which does not require any global queue
length information or explicit knowledge of system parameters,
and yet provides provably near-optimal service elasticity.
We establish the fluid-level dynamics for the proposed scheme
in a regime where the total traffic volume and nominal service
capacity grow large in proportion.
The fluid-limit results show that the proposed scheme achieves
asymptotic optimality in terms of user-perceived delay performance
as well as energy consumption.
Specifically, we prove that both the waiting time of tasks
and the relative energy portion consumed by idle servers vanish
in the limit.
At the same time, the proposed scheme operates in a distributed
fashion and involves only constant communication overhead per task,
thus ensuring scalability in massive data center operations.
Extensive simulation experiments corroborate the fluid-limit results,
and demonstrate that the proposed scheme can match the user
performance and energy consumption of state-of-the-art approaches
that do take full advantage of a centralized queue.
\end{abstract}


\maketitle

\section{Introduction}
\label{sec:intro}
\vspace{.25cm}
\noindent
{\bf Background and motivation.}
Over the last two decades, data centers and cloud networks have
evolved into the digital factories of the world.
This economical and technological evolution goes hand in hand
with a pervasive trend
where human lives are increasingly immersed in a digital universe,
sensors and computers generate ever larger amounts of data,
businesses move IT processes to cloud environments,
and network functions are migrated from dedicated systems to shared
infrastructure platforms.
As a result, both the sheer volume and the scope of applications
hosted in data centers and cloud networks continue to expand
at a tremendous rate.
Indeed, a substantial portion of the applications hosted in these systems increasingly have
highly stringent performance requirements in terms of ultra-low
latency and high reliability.
There is strong empirical evidence that 100 ms delay can have
a major adverse impact on ecommerce sales, and just a few ms
latency can have catastrophic consequences for real-time processing
and control functions that are migrated to cloud networks.
In addition, the energy consumption has risen dramatically
and become a dominant factor in managing data center operations
and cloud infrastructure platforms.
The energy consumption in US data centers is estimated to be around
70~million MegaWatt hours annually, the equivalent of 6~million homes,
which has not only made cooling a challenging issue,
but also carries immense financial and environmental cost.

A crucial challenge in the above context is to achieve efficient
server utilization and limit energy consumption while providing excellent user-perceived performance
in the presence of uncertain and time-varying demand patterns.
This is strongly aligned with the critical notion of service elasticity,
which is at the heart of cloud technology and network virtualization.
Service elasticity hinges on the basic premise that the sheer
amount of available resources is abundant, and not likely to act as
a bottleneck in any practical sense.
Thus the key objective is to dynamically scale the amount
of resources that are actively utilized with the actual observed
load conditions so as to curtail cost and energy consumption,
while satisfying certain target performance criteria.
Achieving ideal service elasticity is highly challenging,
since ramping up service capacity involves a significant time lag
due to the lengthy setup period required for activating servers,
which typically exceeds the latency tolerance of real-time
processing and control functions by orders-of-magnitude.
This can be countered by keeping an ample number of idle servers on,
which however would result in substantial cost and energy wastage.
As a further issue that adds to the above challenge, scalability
requires low implementation overhead and minimal state exchange,
especially in distributed systems with huge numbers of servers.

Auto-scaling provides a popular paradigm for automatically
adjusting service capacity in response to fluctuating demand,
and is widely deployed by major industry players like Amazon Web
Services, Facebook, Google and Microsoft Azure.
While some auto-scaling approaches are primarily predictive
in nature, and use load forecasts based on historical records,
more advanced mechanisms that have been proposed in the literature
operate in a mostly reactive manner.
The latter mechanisms exploit actual load measurements
or state information from a centralized queue to dynamically
activate or deactivate servers, and are inherently better suited
to handle unpredictable load variations.
In typical data center architectures and cloud environments however,
no centralized queue is maintained, and load balancing algorithms
immediately distribute incoming tasks among parallel queues.
In these distributed settings with vast numbers of servers, centralized queue-driven auto-scaling
techniques involve a substantial communication overhead and major
implementation burden, or may not even be viable at all. 
Indeed, even if global queue length information could be gathered, the lack of a centralized queueing operation implies that the overall system is not work-conserving, i.e., some servers may be idling while tasks are waiting at other servers.  Aside from the communication overhead, it hence remains unclear what performance to expect in non-work-conserving scenarios from auto-scaling techniques designed for a centralized queue.\\

\noindent
{\bf Key contributions.}
Urged by the above observations, we propose in the present paper
a joint auto-scaling and load balancing scheme which does not require
any global queue length information or explicit knowledge of system
parameters, and yet achieves near-optimal service elasticity.
We consider a scenario as described above where arriving tasks must
instantaneously be dispatched to one of several parallel servers.
For convenience, we focus on a system with just a single dispatcher,
but the proposed scheme naturally extends to scenarios with
multiple dispatchers.

The proposed scheme involves a token-based feedback protocol,
allowing the dispatcher to keep track of idle-on servers in standby mode as well as
servers in idle-off mode or setup mode.
Specifically, when a server becomes idle, it sends a message to the
dispatcher to report its status as idle-on.
Once a server has remained continuously idle for more than
an exponentially distributed amount of time with parameter $\mu>0$
(standby period), it turns off, and sends a message to the
dispatcher to change its status to idle-off.

When a task arrives, and there are idle-on servers available,
the dispatcher assigns the task to one of them at random,
and updates the status of the corresponding server to busy accordingly.
Otherwise, the task is assigned to a randomly selected busy server.
In the latter event, if there are any idle-off servers,
the dispatcher instructs one of them at random to start the setup
procedure, and updates the status of the corresponding server from idle-off to setup mode. 
It then takes an exponentially distributed amount of time with
parameter $\nu>0$ (setup period) for the server to become on,
at which point it sends a message to the dispatcher to change its
status from setup mode to idle-on.

Note that tasks are only dispatched to `on' servers (idle or busy), and in no
circumstance assigned to an `off' server (idle-off or setup mode).
Also, a server only sends a (green, say) message when a task
completion leaves its queue empty, and sends at most one (red, say)
message when it turns off after a standby period per green message,
so that at most two messages are generated per task.


In order to analyze the response time performance and energy consumption of the
proposed scheme, we consider a scenario with $N$ homogeneous servers,
and  establish the fluid-level dynamics for the proposed scheme
in a regime where the total task arrival rate and nominal number
of servers grow large in proportion. 
This regime not only offers analytical tractability, but is also highly relevant given the massive numbers of servers in data centers and cloud networks.
The fluid-limit results show that the proposed scheme achieves
asymptotic optimality in terms of response time performance
as well as energy consumption.
Specifically, we prove that for any positive values of~$\mu$
and~$\nu$ both the waiting time incurred by tasks and the relative
energy portion consumed by idle servers vanish in the limit.
The latter results not only hold for exponential service time distributions, but also extend to a multi-class scenario with phase-type service time distributions.
To the best of our knowledge, this is the first scheme to provide auto-scaling capabilities in a setting with distributed queues and achieve near-optimal service elasticity.
Extensive simulation experiments corroborate the fluid-limit results,
and demonstrate that the proposed scheme can match the user
performance and energy consumption of state-of-the-art approaches
that do assume the full benefit of a centralized queue.\\

\noindent
{\bf Discussion of related schemes and further literature.}
As mentioned above, centralized queue-driven auto-scaling
mechanisms have been widely considered in the literature~\cite{ALW10, GDHS13, LCBWGWMH12, LLWLA11a, LLWLA11b, LLWA12, LWAT13, PP16, UKIN10, WLT12}.
Under Markovian assumptions, the behavior of these mechanisms can
be described in terms of various incarnations of M/M/N queues
with setup times.
A particularly interesting variant  considered  
by Gandhi \emph{et~al.}~\cite{GDHS13} is referred to as M/M/N/setup/delayedoff.
In this mechanism, when a server $s$ finishes a service, and finds no immediate waiting task, it waits for an exponentially distributed amount of time with parameter $\mu$.
In the meantime, if a task arrives, then it is immediately assigned to server $s$ (or one of the idle-on servers at random), otherwise server $s$ is turned off.
When a task arrives, if there is no idle-on server, then it selects one of the switched off servers $s'$ say (if any), starts the setup procedure in $s'$, and waits in the queue for service. 
The setup procedure also takes an exponentially distributed amount of time with parameter $\nu$.
During the setup procedure, if some other server completes a service, then the waiting task at the head of the queue is assigned to that server, and the server $s'$ terminates its setup procedure unless there is any task $w$ waiting in the queue that had not started a setup procedure (due to unavailability of idle-off servers at its arrival epoch).
In the latter event, the server continues to be in setup mode for task $w$.
Gandhi \emph{et~al.}~\cite{GDHS13} provide an exact analysis of this model, and observe that this mechanism performs very well in a work-conserving pooled server scenario.
There are several further recent papers which examine on-demand server addition/removal in a somewhat different vein~\cite{PS16, NS16}. 
 Generalizations towards non-stationary arrivals and impatience effects have also been considered recently~\cite{PP16}.

Another related strand of research that starts from the seminal paper~\cite{YDS95} is concerned with scaling the speed of a single processor in order to achieve an optimal trade-off between energy consumption and response time performance. 
In this framework, a stream of tasks having specific deadlines arrive at a processor that either accepts the task and finishes serving it before the deadline, or discards the task at arrival.
The processor can work faster at the cost of producing more heat.
To strike the optimal balance between the revenue earned due to task completions and the energy usage, the server can scale its speed, (possibly) depending on its current load.
Dynamic versions of this speed-scaling scenario have been studied in~\cite{BPS07, B05, C72, WS87, WLT12}
A further research direction \cite{ALW10, LLWA12, LWAT13, LCBWGWMH12, LLWLA11a, LLWLA11b} considers online algorithms  for the use of green-energy sources distributed across  geographically different locations that meet the energy demands and reduce expensive energy storage capacity.

In case standby periods are infinitely long, idle servers always remain active and the proposed scheme
corresponds to the so-called Join-the-Idle-Queue (JIQ) policy,
which has gained huge popularity recently \cite{BB08,LXKGLG11}.
In the JIQ policy, idle servers send tokens to the dispatcher
to advertize their availability.
When a task arrives and the dispatcher has tokens available,
it assigns the task to one of the corresponding servers
(and disposes of the token).
When no tokens are available at the time of a task arrival, the
task is simply dispatched to a randomly selected server.

Fluid-limit results in \cite{S15,S15mult} show that
under Markovian assumptions, the JIQ policy achieves a zero
probability of wait for any fixed subcritical load per server
in a regime where the total number of servers grows large.
Results in~\cite{MBLW15} indicate that the JIQ policy exhibits the
same diffusion-limit behavior as the Join-the-Shortest-Queue (JSQ)
strategy, and thus achieves optimality at the diffusion level.
These results show that the JIQ policy provides asymptotically
optimal delay performance while only involving minimal
communication overhead (at most one message per task).
However, in the JIQ policy no servers are ever deactivated,
resulting in a potentially excessive amount of energy wastage.
The scheme that we propose retains the low communication overhead
of the JIQ policy (at most two messages per task) and also
preserves the asymptotic optimality at the fluid level,
in the sense that the waiting time vanishes in the limit.
At same time, however, any surplus idle servers are judiciously
deactivated in our scheme, ensuring that the relative energy
wastage vanishes in the limit as well. \\

\noindent
{\bf Organization of the paper.}
The remainder of the paper is organized as follows.
In Section~\ref{sec:model} we present a detailed model description, and provide
a specification of the proposed scheme.
In Section~\ref{sec:results} we state the main results, and offer an interpretation
and discussion of their ramifications with
the full proof details relegated to Section~\ref{sec:proofs}. 
In Section~\ref{sec:phase-type} we describe how the fluid-limit results extend to phase-type service time distributions.
In Section~\ref{sec:simulation} we discuss the simulation experiments that we
conducted to support the analytical results and to benchmark the
proposed scheme against state-of-the-art approaches.
We make a few brief concluding remarks and offer some suggestions
for further research in Section~\ref{sec:conclusion}.

\section{Details of model and Algorithm}\label{sec:model}
\vspace{.25cm}
\noindent
{\bf Model description.}
Consider a system of $N$~parallel queues with identical servers and a single dispatcher. 
Tasks with unit-mean exponentially distributed service requirements arrive as a Poisson process of rate $\lambda_N(s) = N\lambda(s)$ at time $s\geq 0$, where $\lambda(\cdot)$ is a bounded positive real-valued function, bounded away from zero.
In case of a fixed arrival rate, $\lambda(s)\equiv \lambda$ is assumed to be constant. 
Incoming tasks cannot be queued at the dispatcher, and must immediately and irrevocably be forwarded to one of the servers where they can be queued, possibly subject to a finite buffer capacity limit $B$. 
The service discipline at each server is oblivious to the actual service requirements (e.g., FCFS).
A turned-off server takes an Exp$(\nu)$ time (setup period) to be turned on.

We now introduce a token-based joint auto-scaling and load balancing scheme called TABS (Token-based Auto Balance Scaling).\\

\noindent
{\bf Algorithm specification.} TABS:
\begin{itemize}
\item When a server becomes idle, it sends a `green' message to the dispatcher, waits for an $\expn(\mu)$ time (standby period), and turns itself off by sending a `red' message to the dispatcher (the corresponding green message is destroyed).
\item When a task arrives, the dispatcher selects a green message at random if there are any, and assigns the task to the corresponding server (the corresponding green message is replaced by a `yellow' message). 
Otherwise, the task is assigned to an arbitrary busy  server, and if at that arrival epoch there is a red message at the dispatcher, then it selects one at random, and the setup procedure of the corresponding server is initiated, replacing its red message by an `orange' message.
\item Any server which activates due to the latter event, sends a green message to the dispatcher (the corresponding orange message is replaced), waits for an $\expn(\mu)$ time for a possible assignment of a task, and again turns itself off by sending a red message to the dispatcher.
\end{itemize}
\begin{figure}
\begin{center}
\includegraphics[scale=.8]{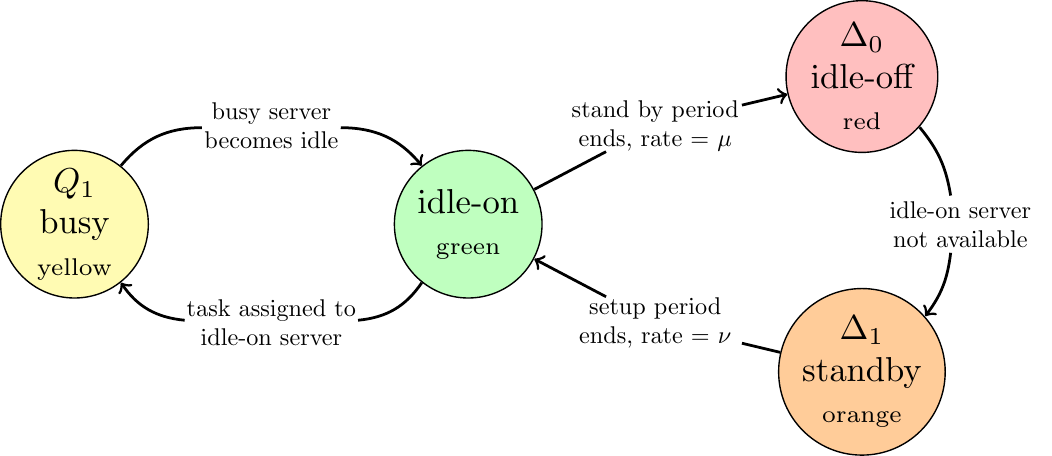}
\end{center}
\caption{Illustration of server on-off decision rules in the TABS scheme, along with message colors and state variables.}
\label{fig:scheme}
\end{figure}
The TABS scheme gives rise to a distributed operation in which  servers are in one of four states (busy, idle-on, idle-off or standby), and advertize their state to the dispatcher via exchange of tokens. Figure~\ref{fig:scheme} illustrates this token-based exchange protocol. 
Note that setup procedures are never aborted and continued even when idle-on servers do become available.  When setup procedures \emph{are} terminated in the latter event, the proposed scheme somewhat resembles the delayed-off scheme considered by Gandhi \emph{et~al.}~\cite{GDHS13} in terms of auto-scaling actions.  This comes however with an extra overhead penalty, without producing any improvement in response time performance or energy consumption in the large-capacity limit, as will be shown later.  \\

\noindent
{\bf Notation.}
Let 
$$\mathbf{Q}^N(t) := (Q_1^N(t), Q_2^N(t), \dots, Q_B^N(t))$$ denote the system occupancy state, where $Q_i^N(t)$ is the number of servers with queue length greater than or equal to $i$ at time $t$, including the possible task in service.
Also, let $\Delta_0^N(t)$ and $\Delta_1^N(t)$ denote the number of idle-off servers and servers  in setup mode at time $t$, respectively. 
Note that the process $(\mathbf{Q}^N(t),\Delta_0^N(t),\Delta_1^N(t))_{t\geq 0}$ provides a proper state description by virtue of the exchangeablity of the servers and is Markovian.
The exact analysis of the above system becomes complicated due to the strong dependence among the queue length processes of the various servers.
Moreover, the arrival processes at individual servers are not renewal processes, which makes the problem even more challenging.
Thus we resort to an asymptotic analysis, where the task arrival rate and number of servers grow large in proportion.
In the limit the collective system then behaves like a deterministic system, which is amenable to analysis.
The fluid-scaled quantities are denoted by the respective small letters, \emph{viz.}~$q_i^{N}(t):=Q_i^{N}(t)/N$, $\delta_0^N(t) = \Delta_0^N(t)/N$, and $\delta_1^N(t) = \Delta_1^N(t)/N$. For brevity in notation, we will write $\mathbf{q}^N(t) = (q_1^N(t),\dots,q_B^N(t))$ and $\bld{\delta}^N(t) = (\delta_0^N(t),\delta_1^N(t))$. 
Let
\begin{align*}
E = \big\{(\bld{q},\dd)\in [0,1]^{B+2}:  q_i\geq q_{i+1},\ \forall i, \  \delta_0+\delta_1+\sum_{i=1}^Bq_i\leq 1 \big\},
\end{align*}
denote the space of all fluid-scaled occupancy states,
so that $(\mathbf{q}^N(t),\bld{\delta}^N(t))\in E$ for all $t$.
Endow $E$ with the product topology, and the Borel $\sigma$-algebra $\mathcal{E}$, generated by the open sets of $E$.
For stochastic boundedness of a process we refer to \cite[Definition 5.4]{PTRW07}. 
For any complete separable metric space $E$, denote by $D_E[0,\infty)$, the set of all $E$-valued \emph{c\`adl\`ag} (right continuous with left limit exists) processes.
By the symbol `$\dto$' we denote weak convergence for real-valued random variables, and convergence with respect to Skorohod-$J_1$ topology for c\`adl\`ag processes.

\section{Overview of Results}
\label{sec:results}
In this section we provide an overview of the main results and discuss their ramifications.  For notational transparency, we focus on the case of exponential service time distributions.  In Section~\ref{sec:phase-type} we show how some of the results extend to phase-type service time distributions, at the expense of more complex notation.
\begin{theorem}[{Fluid limit for exponential service time distributions}]
\label{th: fluid}
Assume that $(\qq^N(0),\bld{\delta}^N(0))$ converges weakly to $(\qq^\infty,\bld{\delta}^\infty)\in E$, as $N\to\infty$, where $q_1^\infty>0$. Then the process $\{(\qq^N(t),\dd^N(t))\}_{t\geq 0}$ converges weakly to the deterministic process $\{(\qq(t),\dd(t))\}_{t\geq 0}$ as $N\to\infty$, which satisfies the following integral equations:
\begin{align*}
 q_i(t)&=q_i^\infty +\int_0^t\lambda(s) p_{i-1}(\qq(s),\dd(s),\lambda(s))\dif s\\
 &\hspace{1cm}- \int_0^t(q_i(s)-q_{i+1}(s))\dif s,\ i= 1,\ldots,B, \\
\delta_0(t)&=\delta_0^\infty+\mu\int_0^t u(s)\dif s-\xi(t),  \\
\delta_1(t)&=\delta_1^\infty+\xi(t)-\nu\int_0^t\delta_1(s)\dif s,\nonumber
\end{align*}
where by convention $q_{B+1}(\cdot) \equiv 0$,
and
\begin{align*}
u(t) &= 1- q_1(t) - \delta_0(t) - \delta_1(t),\\
\xi(t)& = \int_0^t\lambda(s)(1-p_0(\qq(s),\dd(s),\lambda(s)))\ind{\delta_0(s)>0}\dif s.
\end{align*}
For any $(\qq,\dd)\in E$, $\lambda>0$, $(p_i(\qq,\dd,\lambda))_{i\geq 0}$ are given by 
\begin{align*}
 p_0(\qq,\dd,\lambda) &= 
\begin{cases}
&1\qquad \text{if}\qquad u=1-q_1-\delta_0-\delta_1>0,\\
&\min \{\lambda^{-1}(\delta_1\nu + q_1-q_2), 1\},\quad\text{otherwise,}
\end{cases}\\
\quad p_i(\qq,\dd,\lambda)  &= (1-p_0(\qq,\dd,\lambda)) (q_{i}-q_{i+1})q_1^{-1},\  i =1,\ldots,B.
\end{align*}
\end{theorem}

We now provide an intuitive explanation of the fluid limit stated above.
The term $u(t)$ corresponds to the asymptotic fraction of idle-on servers in the system at time $t$, and $\xi(t)$ represents the asymptotic cumulative number of server setups (scaled by $N$) that have been initiated during $[0,t]$.
The coefficient $p_i(\qq,\dd,\lambda)$ can be interpreted as the instantaneous fraction of incoming tasks that are assigned to some server with queue length $i$, when the fluid-scaled occupancy state is $(\qq,\dd)$ and the scaled instantaneous arrival rate is $\lambda$.
Observe that as long as $u>0$, there are idle-on servers, and hence all the arriving tasks
will join idle servers. 
This explains that if $u>0$, $p_0(\qq,\dd,\lambda) = 1$ and $p_i(\qq,\dd,\lambda)=0$ for $i=2,\ldots,B$.
If $u=0$, then observe that 
servers become idle at rate $q_1-q_2$, and servers in setup mode turn on at rate $\delta_1\nu$.
Thus the  idle-on servers are created at a total rate $\delta_1\nu + q_1-q_2$.
If this rate is larger than the arrival rate $\lambda$, then almost all the arriving tasks can be assigned to idle servers.
Otherwise, only a fraction $(\delta_1\nu + q_1-q_2)/\lambda$
of arriving tasks join idle servers. 
The rest of the tasks are distributed uniformly among busy servers, so a proportion $(q_{i}-q_{i+1})q_1^{-1}$ are assigned to servers having queue length~$i$.
For any $i=1,\ldots,B$, $q_i$ increases when there is an arrival to some server with queue length $i-1$, which occurs at rate $\lambda p_{i-1}(\qq,\dd,\lambda)$, and it decreases when there is a departure from some server with  queue length $i$, which occurs at rate $q_i-q_{i-1}$. 
Since each idle-on server turns off at rate $\mu$, the fraction of servers in the off mode increases at rate 
$\mu u$.
Observe that if $\delta_0>0$, for each task that cannot be assigned to an idle server, a setup procedure is initiated  at one idle-off server. 
As noted above, $\xi(t)$ captures the (scaled) cumulative number of setup procedures initiated up to time~$t$.
Therefore the fraction of idle-off servers and the fraction of servers in setup mode decreases and increases by $\xi(t)$, respectively, during $[0,t]$.
Finally, since each server in setup mode becomes idle-on at rate $\nu$, the fraction of servers in setup mode decreases at rate $\nu\delta_1$.\\

\noindent
{\bf Fixed point.}
In case of a constant arrival rate $\lambda(t)\equiv\lambda<1$, the fluid limit in Theorem~\ref{th: fluid} has a unique fixed point:
\begin{equation}\label{eq:fixed point}
\delta_0^*=1-\lambda,\qquad\delta_1^\star=0,\qquad q_1^*=\lambda\quad\mbox{and}\quad q_i^*=0,
\end{equation}
for $i=2,\ldots,B.$ 
Indeed, it can be verified that $p_0(\qq^*,\dd^*,\lambda)= 1$ and $u^*=0$ for $(\qq^*,\dd^*)$ given by~\eqref{eq:fixed point} so that the derivatives of $q_i$, $i = 1,\dots,B$, $\delta_0$, and $\delta_1$ become zero, and that these cannot be zero at any other point in $E$.
Note that, at the fixed point, a fraction $\lambda$ of the servers have exactly one task while the remaining fraction have zero tasks, independently of the values of the parameters $\mu$ and $\nu$.

The next proposition states the global stability of the fluid limit, i.e., starting from any point in $E$, the dynamical system  defined by the system of integral equations in Theorem~\ref{th: fluid} converges to the fixed point~\eqref{eq:fixed point} as $t\to\infty$.
\begin{proposition}[{Global stability of the fluid limit}]
\label{prop:glob-stab}
Assume $(\qq(0), \dd(0)) = (\qq^\infty,\dd^\infty)\in E$. Then  
$$(\qq(t),\dd(t))\to (\qq^*,\dd^*),\quad \mbox{as}\quad t\to\infty,$$
where $(\qq^*,\dd^*)$ is as defined in~\eqref{eq:fixed point}.
\end{proposition}
There are general methods to prove global stability if the evolution of the dynamical system  satisfies some kind of monotonicity property induced by the drift structure~\cite{TX11, Mitzenmacher01}.
Here, it is not straightforward to establish such a monotonicity property, and harder to find a suitable Lyapunov function.
Instead we exploit specific properties of the fluid limit in order to prove the global stability.
Observe that the global stability in particular also establishes the uniqueness of the fixed point above. 
The proof of Proposition~\ref{prop:glob-stab} is presented in Section~\ref{sec:proofs}.

The global stability can be leveraged to show that the steady-state distribution of the $N^{\mathrm{th}}$ system, for large $N$, can be well approximated by the fixed point of the fluid limit in~\eqref{eq:fixed point}. 
Specifically, in the next proposition, whose proof we provide in the appendix, we demonstrate the convergence of the steady-state distributions, and hence the interchange of the large-capacity
($N\to\infty$) and steady-state ($t\to\infty$) limits.
Since the buffer capacity $B$ at each server is finite, for each $N$, the Markov process $(\QQ^N(t), \Delta_0^N(t),\Delta_1^N(t))$ is irreducible, has a finite state space, and thus has a unique steady-state distribution.
Let $\pi^N$ denote the steady-state distribution of the $N^{\mathrm{th}}$ system, i.e.,
$$\pi^{N}(\cdot)=\lim_{t\to\infty}\Pro{\qq^{N}(t)=\cdot, \dd^N(t) = \cdot}.$$ 
\begin{proposition}[{Interchange of limits}]
\label{thm:limit interchange}
As $N\to\infty$,
$\pi^N\dto\pi$, where $\pi$ is given by the Dirac mass concentrated upon $(\qq^\star,\dd^\star)$ defined in~\eqref{eq:fixed point}.
\end{proposition}

\noindent
{\bf Performance metrics.}
As mentioned earlier, two key performance metrics are the expected waiting time of tasks $\expt[W^N]$ and energy consumption $\expt[ P^N]$ for the $N^{\mathrm{th}}$ system in steady state.
In order to quantify the energy consumption, we assume that the energy usage of a server is $P_{\full}$ when busy or in set-up mode, $P_{\idle}$ when idle-on, and zero when turned off.
Evidently, for any value of $N$, at least a fraction $\lambda$ of the servers must be busy in order for the system to be stable, and hence $\lambda P_{\full}$ is the minimum mean energy usage per server needed for stability.
We will define $\expt[Z^N]=\expt[P^N]-\lambda P_{\full}$ as the relative energy wastage accordingly.
The next proposition demonstrates that asymptotically the expected waiting time and  energy consumption for the TABS scheme vanish in the limit, for any strictly positive values of $\mu$ and $\nu$.
The key implication is that
the TABS scheme, while only involving constant communication overhead per task, provides performance in a distributed setting that is as good at the fluid level as can possibly be achieved, even in a centralized queue, or with unlimited information exchange.
\begin{proposition}[Asymptotic optimality of TABS scheme]
\label{prop:perform}
In a fixed arrival rate scenario $\lambda(t)\equiv \lambda<1$, for any $\mu>0$, $\nu>0$, as $N\to\infty$,
\begin{enumerate}[{\normalfont (a)}]
\item\ $[$zero mean waiting time$]$ $\expt[W^N]\to 0$,
\item\ $[$zero energy wastage$]$ $\expt[Z^N]\to 0$.
\end{enumerate}
\end{proposition}

\begin{proof}[Proof of Proposition~\ref{prop:perform}]
By Little's law, the mean stationary waiting time $\expt[W^N]$ in the $N^{\mathrm{th}}$ system may be expressed as $(N \lambda)^{-1} \expt[L^N]$, where $L^N = \sum_{i = 2}^{B} Q_i^N$ represents a random variable with the stationary distribution of the total number of waiting tasks in the $N^{\mathrm{th}}$ system.
Thus, $\expt[W^N] = \lambda^{-1} \sum_{i = 2}^{B} q_i^N$, where $\qq^N$ is a random vector with the stationary distribution of $\qq^N(t)$ as $t \to \infty$.
Invoking Proposition~\ref{thm:limit interchange} and the fixed point as identified in~\eqref{eq:fixed point}, we obtain that $\expt[W^N] \to \sum_{i = 2}^{B} q_i^* = 0$ as $N \to \infty$.

Denoting by $U^N= N-Q_1^N-\Delta_0^N-\Delta_1^N$ the number of idle-on servers,
the stationary mean energy consumption per server in the $N^{\mathrm{th}}$ system may  be expressed as 
$$\frac{1}{N} \expt[(Q_1^N + \Delta_1^N) P_{\full} + U^N P_{\idle}] = \expt[(q_1^N + \delta_1^N) P_{\full} + u^N P_{\idle}].$$
Applying Proposition~\ref{thm:limit interchange} and the fixed point as identified in~\eqref{eq:fixed point}, we deduce that $\expt[P^N] \to (q_1^* + \delta_1^*) P_{\full} + u^* P_{\idle} = (1 - \delta_0^*) P_{\full} - u^* (P_{\full} - P_{\idle}) = \lambda P_{\full}$ as $N \to \infty$.
This yields that $\expt[Z^N] = \expt[P^N]-\lambda P_{\full}\to0.$
\end{proof}

 The quantitative values of the energy usage and waiting time for finite values of $N$ will be evaluated through extensive simulations in Section~\ref{sec:simulation}.\\

\noindent
{\bf Comparison to ordinary JIQ policy.}
Consider the fixed arrival rate scenario $\lambda(t) \equiv \lambda$. 
It is worthwhile to observe that the component $\qq$ of the fluid limit in Theorem~\ref{th: fluid} coincides with that for the ordinary JIQ policy where servers always remain on, when the system starts with all the servers being idle-on, and $\lambda+\mu<1$. 
To see this, observe that the component $\qq$ depends on $\dd$ only through $(p_{i-1}(\qq,\dd))_{i\geq 1}$. 
Now, $p_0 =1$, $p_i = 0$, for all $i\geq 1$, whenever $q_1+\delta_0+\delta_1<1$, irrespective of the precise values of $(\qq,\dd)$. 
Moreover, starting from the above initial state, $\delta_1$ can increase only when $q_1+\delta_0=1$. 
Therefore, the fluid limit of $\qq$ in Theorem~\ref{th: fluid} and the ordinary JIQ scheme are identical if the system parameters $(\lambda,\mu,\nu)$ are such that $q_1(t)+\delta_0(t) < 1$, for all $t\geq 0$. Let $y(t)  = 1-q_1(t) - \delta_0(t)$. The solutions to the differential equations
\begin{equation*}
 \frac{\dif q_1(t)}{\dif t} = \lambda - q_1(t), \quad  \frac{\dif y(t)}{\dif t} = q_1(t) - \lambda -\mu y(t),
\end{equation*}$y(0)=1$, $q_1(0) = 0$ are given by 
\begin{equation*}
 q_1(t)= \lambda(1-\e^{-t}),\quad y(t) = \frac{\e^{-(1+\mu)t}}{\mu-1}\big(\e^t(\lambda+\mu-1)-\lambda \e^{\mu t}\big).
\end{equation*}
Notice that if $\lambda+\mu<1$, then $y(t) > 0$ for all $t\geq 0$ and thus, $q_1(t)+\delta_0(t) < 1$, for all $t\geq 0$.
The fluid-level optimality of the JIQ scheme was shown in~\cite{S15,S15mult}. 
This observation thus establishes the optimality of the fluid-limit trajectory under the TABS scheme for suitable parameter values in terms of response time performance.
From the energy usage perspective, under the ordinary JIQ policy, since 
the asymptotic steady-state fraction of busy servers ($q_1^*$) and idle-on servers are given by $\lambda$ and $1-\lambda$, respectively, the asymptotic steady-state (scaled) energy usage is given by 
\begin{align*}
\expt[P^{\mathrm{JIQ}}] 
 = \lambda P_{\full} + (1-\lambda) P_{\idle} 
= \lambda P_{\full}(1+ (\lambda^{-1}-1)f),
\end{align*}
where $f = P_{\idle}/P_{\full}$ is the relative energy consumption of an idle server.
Proposition~\ref{prop:perform} implies that the asymptotic steady-state (scaled) energy usage under the TABS scheme is $\lambda P_{\full}.$
Thus the TABS scheme reduces the asymptotic steady-state energy usage by $\lambda P_{\full}(\lambda^{-1}-1)f = (1-\lambda)P_{\idle},$ which amounts to a relative saving of $(\lambda^{-1}-1)f/(1+(\lambda^{-1}-1)f).$ 
In summary, the TABS scheme performs as good as the ordinary JIQ policy in terms of the waiting time and communication overhead while providing a significant energy saving.

\section{Extension to phase type service time distributions}
\label{sec:phase-type}

In this section we extend the fluid-limit results to phase type service time distributions. 
Specifically, the service time of each task is described by a time-homogeneous,  continuous-time Markov process with a finite state space $\{0,1,\dots,K\}$, initial distribution $\bld{r} = (r_i:0\leq i\leq K)$, transition probability matrix  $R = (r_{i,j})$, and the mean sojourn time in state $i$ being $\gamma_i^{-1}$. 
State 0 is an absorbing state, and thus  represents a service completion, while state $j$ is referred to as a type-$j$ service, and is assumed to be transient. 
For convenience, and without loss of generality, it is assumed that $r_{i,i} = 0$ for all $i$, and that any incoming task  has a non-zero service time ($r_0 =0$).  
Consider a time-homogeneous discrete-time Markov chain with the state  space $\{0,1,\dots,K\}$, and transition probability matrix $P=(p_{i,j})$, where $p_{i,j} = r_{i,j}$ for $i\geq 1$, $p_{0,j} = r_j$ $j\geq 1$, and $p_{0,0}=0$. Let $\bld{\eta}= (\eta_0,\dots,\eta_K)$ be the stationary distribution, i.e., $\bld{\eta}$ satisfies
\begin{equation}
 \eta_0r_i + \sum_{j=1}^Kr_{j,i}\eta_j = \eta_i, \quad i\geq 1,\quad \sum_{i=0}^K \eta_i = 1.
\end{equation}
The mean of the phase type service time distribution~\cite{PR00} is $(\sum_{i=1}^K\eta_i/\gamma_i\eta_0)^{-1}$, and is assumed to be one.

We assume now that the service discipline at each server is not only oblivious of the actual service requirements, but also non-preemptive, and allows at most one task to be served at any given time.
Let $Q_{i,j}^{N}(t)$ denote the number of servers with queue length at least $i$ and providing a type-$j$ service at time~$t$. Thus, $Q_i^N(t) = \sum_{j=1}^KQ_{i,j}^N(t)$. 
Denote the fluid-scaled quantities by $q_{i,j}^N(t) = Q_{i,j}^{N}(t)/N$ and the vector $\qq^N(t) = (q_{i,j}^N(t):1\leq i\leq B , 1\leq j\leq K)$. 
Let $\delta_0^N(t)$ and $\delta^N_1(t)$ be as defined before.
Let 
\begin{align*}
\hat{E} &= \bigg\{\big((q_{i,j})_{1\leq i\leq B, 1\leq j\leq K}, (\delta_0,\delta_1)\big): q_{1,j},\delta_0,\delta_1\in [0,1],\\
&\hspace*{.6cm} q_{i+1,j}\leq q_{i,j},\ \forall i,j, \quad
 \delta_0 + \delta_1 + \sum_{j=1}^K q_{1,j}\leq 1
\bigg\}
\end{align*}
denote the space of all fluid-scaled occupancy states,
so that $(\mathbf{q}^N(t),\bld{\delta}^N(t))\in \hat E$ for all $t$, and as before,
endow $\hat{E}$ with the product topology, and the Borel $\sigma$-algebra $\hat{\mathcal{E}}$, generated by the open sets of $\hat E$.
\begin{theorem}[Fluid limit for phase type service time distributions]
\label{th: fluid gen service}
Assume that $(\qq^N(0),\bld{\delta}^N(0))$ converges weakly to $(\qq^\infty,\bld{\delta}^\infty)\in \hat{E}$, as $N\to\infty$, where $\sum_{j=1}^Kq_{1,j}^{\infty}>0$. Then the sequence of processes $\{\qq^N(t),\dd^N(t)\}_{t\geq 0}$ converges weakly to the  deterministic process $\{\qq(t),\dd(t)\}_{t\geq 0}$, as $N\to\infty$, which satisfies the following integral equations: for $i = 1,\ldots,B$ and $j = 1,\ldots,K$,
\begin{align}\label{eq:phasetype-fluid}
q_{i,j}(t)&=q_{i,j}^\infty+\int_0^t\lambda(t) p_{i-1,j}(\qq(s),\dd(s),\lambda(s))\dif s\\
 + \int_0^t&\sum_{k=1}^K (q_{i,k}(s)-q_{i+1,k}(s))\gamma_kr_{k,j}\dif s - \gamma_j\int_0^t q_{i,j}(s)\dif s \nonumber \\ 
  + \int_0^t&\sum_{k=1}^K (q_{i+1, k}(s) - q_{i+2, k}(s))\gamma_kr_{k,0}r_j\dif s,  \nonumber\\
\delta_0(t)&=\delta_0^\infty + \mu\int_0^tu(s)\dif s
 -\xi(t),\\
\delta_1(t) &=\delta_1^\infty + \xi(t)-\nu\int_0^t\delta_1(s)\dif s,
\end{align}
where by convention $q_{B+1,j}(\cdot) \equiv 0$, $j= 1,\ldots,K$,
and
\begin{align}
 u(t) &= 1-\sum_{j=1}^Kq_{1,j}(t)-\delta_0(t)-\delta_1(t),\\
 \xi(t) &= \int_0^t\lambda(s)\bigg(1-\sum_{j=1}^Kp_{0,j}(\qq(s),\dd(s),\lambda(s))\bigg)\ind{\delta_0(s)>0}\dif s.\nonumber
\end{align}
For any $(\qq,\dd)\in\hat{E}$, $\lambda>0$,
$p_{0,j}(\qq,\dd,\lambda) = r_j$ if $u=1- \sum_{j=1}^Kq_{1,j}-\delta_0-\delta_1>0$, $j= 1,\ldots,K$, and otherwise
$$p_{0,j}(\qq,\dd,\lambda) =r_j\min \bigg\{\lambda^{-1}\Big(\delta_1\nu + \sum_{j=1}^K(q_{1,j}-q_{2,j})\gamma_jr_{j,0}\Big), 1\bigg\},$$
and for $i = 1,\dots,B$,
$$p_{i,j} (\qq,\dd,\lambda)= \bigg(1-\sum_{j=1}^Kp_{0,j}(\qq,\dd,\lambda)\bigg)\frac{q_{i-1,j}-q_{i,j}}{\sum_{j=1}^Kq_{1,j}}.$$
\end{theorem}
Let us provide a heuristic justification of the fluid limit stated above. 
As in Theorem~\ref{th: fluid}, 
$u(t)$ corresponds to the asymptotic fraction of idle-on servers in the system at time $t$, 
$\xi(t)$ represents the asymptotic cumulative number of server setups (scaled by $N$) that have been initiated during $[0,t]$. The coefficient $p_{i,j}(\qq(t),\dd(t),\lambda(t))$ can be interpreted as the instantaneous fraction of incoming tasks 
that are assigned to a server with queue length $i\geq 1$ and currently providing a type-$j$ service, 
while $p_{0,j}$ specifies the fraction of incoming tasks assigned to idle servers starting with a type-$j$ service. 
The heuristic justification for the $p_{i,j}$ values builds on the same line of reasoning as for Theorem~\ref{th: fluid}. 
As long as there are idle-on servers, i.e., 
$u>0$, incoming tasks are \emph{immediately} assigned to one of those servers, and the initial service type is chosen according to the distribution $\mathbf{r}$. 
Notice that the busy servers and the servers in setup become idle at total rate $\delta_1\nu + \sum_{j=1}^K(q_{1,j}-q_{2,j})\gamma_jr_{j,0}$. 
For the case when $u=0$, we need to distinguish between two cases, depending on whether 
$\delta_1\nu + \sum_{j=1}^K(q_{1,j}-q_{2,j})\gamma_jr_{j,0} > \lambda$ or not.
In the first case, the incoming tasks are again assigned to idle-on servers immediately.  
However, if $\delta_1\nu + \sum_{j=1}^K(q_{1,j}-q_{2,j})\gamma_jr_{j,0} \leq  \lambda$, then only a fraction $\lambda^{-1}(\delta_1\nu + \sum_{j=1}^K(q_{1,j}-q_{2,j})\gamma_jr_{j,0}$ of the incoming tasks are immediately taken into service. 
In both of the above two subcases, the service types of the incoming tasks follow the distribution $\mathbf{r}$. 
This explains the expression for the $p_{0,j}$ values. 
Also, given that an incoming task does not find an  idle-on server, it is assigned to a server that has queue length $i$ and is currently providing a type-$j$ service with probability $\big(\sum_{j=1}^Kq_{1,j}\big)^{-1}(q_{i-1,j}-q_{i,j})$.
This explains the expression for $p_{i,j}$ for $i\geq 1$.
Now, notice that the expressions for $\delta_0$, and $\delta_1$ remain essentially the same as in Theorem~\ref{th: fluid} due to the fact that the dynamics of $\delta_0$, and $\delta_1$ 
depends on $q_{i,j}$'s only through the fraction of incoming tasks that join an idle-on server, which is determined by the coefficients $p_{0,j}(\qq,\bld{\delta},\lambda)$.
Finally,  $q_{i,j}$ decreases if and only if there is a completion of type-$j$ service at a server with queue length at least $i$. 
Here, we have used the fact $r_{i,i} = 0$.
Now, $q_{i,j}$ can increase due to three events: (i) assignment of an arriving task, which occurs at rate $\lambda p_{i-1,j}(\qq,\dd,\lambda)$, (ii) service completion of some other type, which now requires service of type $j$, and this occurs at rate $\sum_{k}(q_{i,k}-q_{i+1, k})\gamma_k r_{k,j}$, (iii) service completion occurs at some server, the task exits from the system, and the next task at that server starts with a type-$j$ service. 
This occurs at rate $\sum_{k = 1}^K (q_{i+1,k} - q_{i+2,k})\gamma_kr_{k,0}r_j$. \\

\noindent
{\bf Fixed point of the fluid limit.} 
In case of a constant arrival rate $\lambda(t)\equiv\lambda<1$, the unique fixed point of the fluid limit in Theorem~\ref{th: fluid gen service} is given by 
\begin{equation}\label{fix-point-2}
 \delta_0^* = 1-\lambda, \quad \delta_1^* = 0,\quad q_{1,j}^*= \frac{\eta_j}{\eta_0\gamma_j}\lambda,\quad  j=1,\dots,K,
\end{equation}
and $q_{i,j}^*=0$ for all $i=2,\ldots,B$.
Indeed, it can  be verified that the derivatives of $q_{i,j}$, $i=1,\ldots,B$, $j=1,\ldots,K$, $\delta_0$, and $\delta_1$ are zero at $(\qq^*,\dd^*)$ given by~\eqref{fix-point-2}, and that these cannot be zero at any other point in $\hat{E}$. Thus, the fixed point is unique as before.
Notice that in this case also at the fixed point a fraction $\lambda$ of the servers have exactly one task while the remaining fraction have zero tasks,
independent of the values of the parameters $\mu$ and $\nu$, revealing the insensitivity of the asymptotic fluid-scaled steady-state occupancy states to the duration of the standby periods and setup periods.
Further, note that $\sum_{j=1}^K q_{1,j}^* = \lambda$ from the fact that the mean service time is one, irrespective of the initial distribution $\bld{r}$, transition probability matrix $R$, and parameters $\gamma_j$. 
Thus the values of $q_1^*,\ldots, q_B^*$ in the fixed point are insensitive in a distributional sense with respect to the service times.
They only depend on the service time distribution through its mean, and higher-order characteristics like variance have no impact on the steady-state performance in the large capacity limit whatsoever.

\section{Simulation experiments}\label{sec:simulation}
In this section we present extensive simulation results to illustrate the fluid-limit results, and to examine the performance of the proposed TABS scheme in terms of mean waiting time and energy consumption, and compare that with existing strategies.\\

\begin{figure}
\begin{center}
$
\begin{array}{ccc}
\includegraphics[width=78mm]{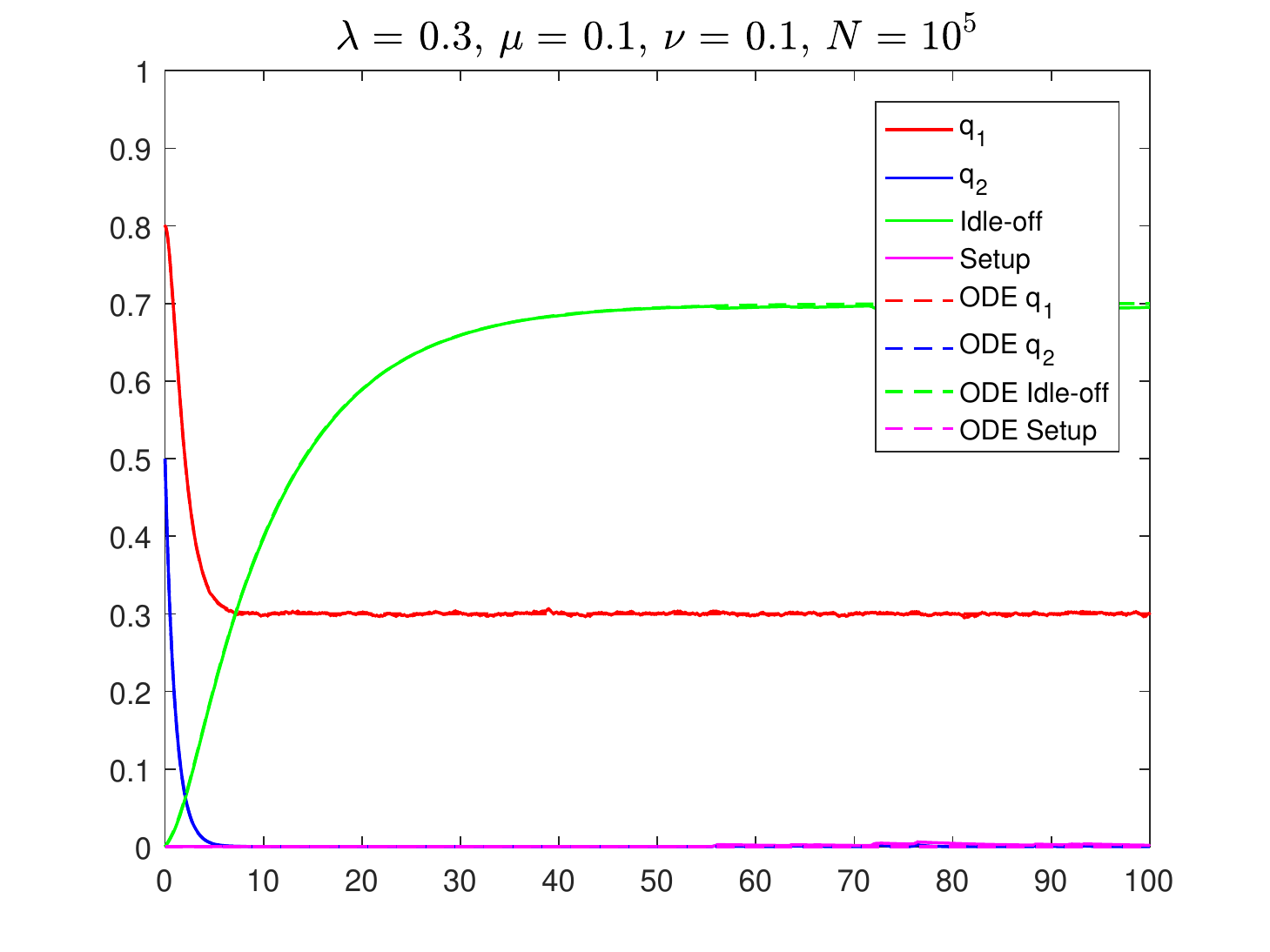}\\
\includegraphics[width=78mm]{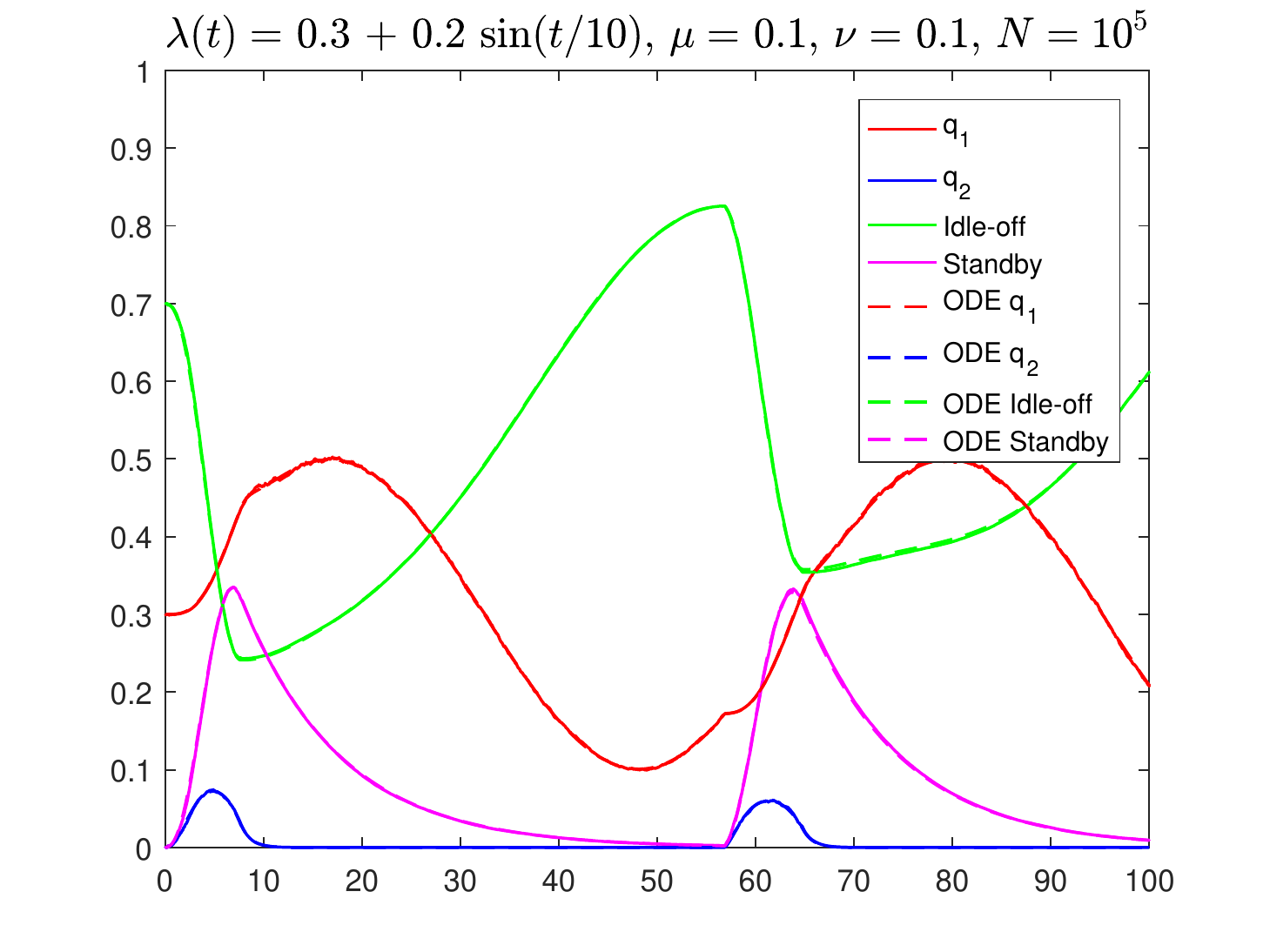}\\
\includegraphics[width=78mm]{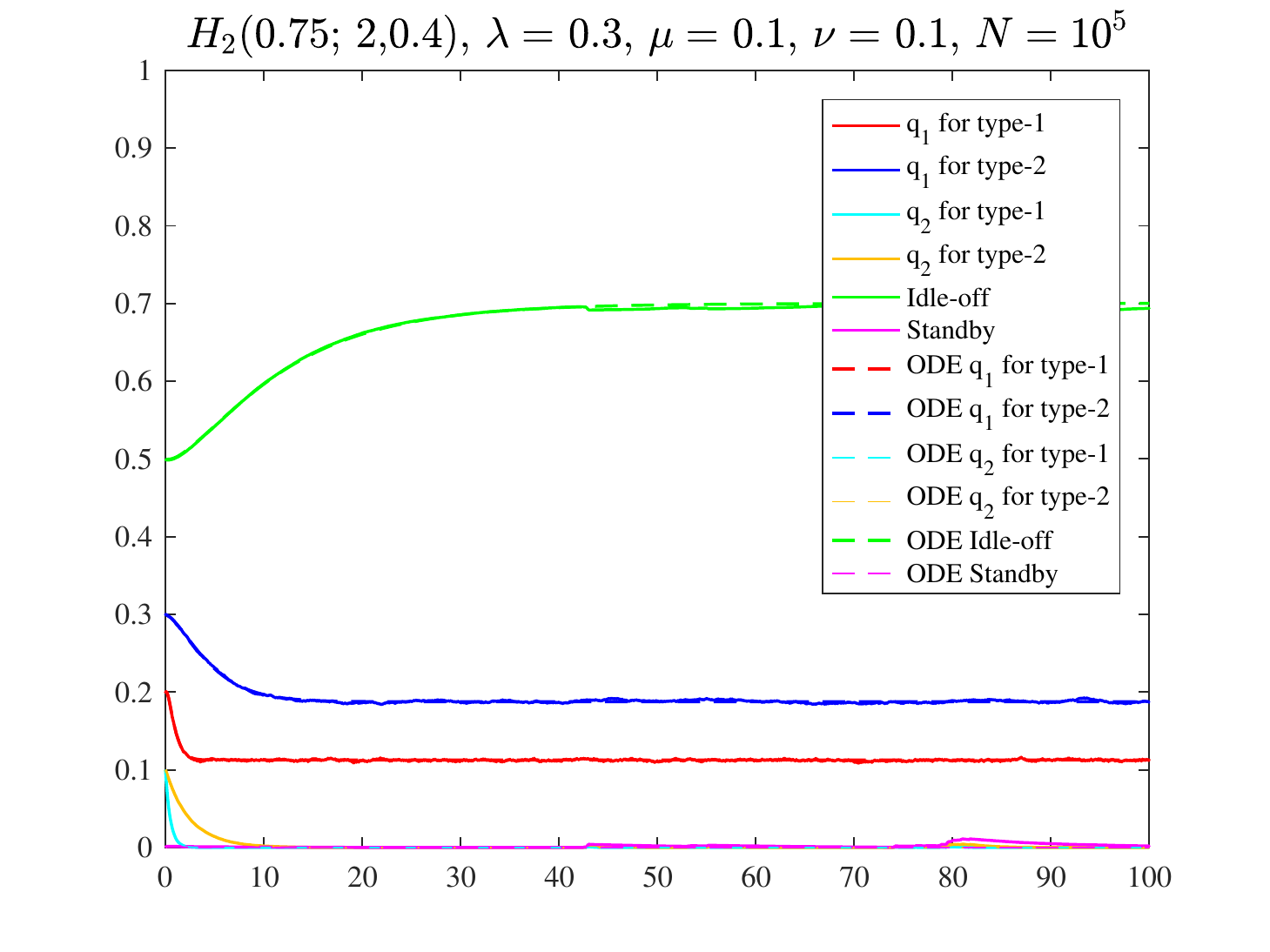}\\
\end{array}$
\end{center}
\caption{Illustration of the fluid-limit trajectories for $N = 10^5$ servers. The left figure is for constant arrival rate $\lambda(t) \equiv 0.3$. The middle figure considers a periodic arrival rate given by $\lambda(t) = 0.3+0.2\sin(t/10)$. The right figure considers a hyper-exponential service time distribution. An incoming task demands  either type-1 or type-2 service with probabilities $0.75$ and $0.25$, respectively. 
The durations of type-1 and type-2 services are exponentially distributed with parameters 2 and 0.4, respectively, and thus the mean service time is 1.}
\label{fig:fluid}
\end{figure}

\noindent
{\bf Convergence of sample paths to fluid-limit trajectories.}
The fluid-limit trajectories for the TABS scheme in Theorems~\ref{th: fluid} and~\ref{th: fluid gen service} are illustrated in Figure~\ref{fig:fluid} for $N=10^5$ servers and three scenarios (constant arrival rate, periodic arrival rate and hyper-exponential service time distribution). 
In all three scenarios the mean standby periods are $\mu^{-1}=10$ and the mean setup periods are $\nu^{-1}=10$. 
In all cases,  the fluid-limit paths and the sample paths obtained from simulation are nearly indistinguishable.
Notice that in case of a time-varying arrival rate the period of fluctuation is only $20\pi\approx 63$ times as long as the mean service time, which is far shorter than what is usually observed in practice. Typically, service times are of sub-second order and variations in the arrival rate occur only over time scales of tens of minutes, if not several hours.
Even in such a challenging scenario, however, the fractions of idle-on servers and those with waiting tasks are negligible. 
In case of the hyper-exponential service time distribution, we note from the bottom figure that the long-term values of $q_1 = q_{1,1} + q_{1,2}, q_{2} = q_{2,1} + q_{2,2}$, $\delta_0$ and $\delta_1$ agree with the corresponding quantities in the top chart for exponential service times.  
This reflects the asymptotic insensitivity in a distributional sense mentioned at the end of Section~\ref{sec:phase-type}, and in particular supports the observation that the proposed TABS scheme achieves asymptotically optimal response time performance and energy consumption for phase-type service time distributions as well.\\

\noindent
{\bf Convergence of steady-state performance metrics to fluid-limit values.}
In order to quantify the energy usage, we will adopt the parameter values from empirical measurements reported in~\cite{GHK12, BH07, GDHS13}. 
A server that is busy or in setup mode, consumes $P_{\full} = 200$ watts,  an idle-on server consumes $P_{\idle} = 140$ watts, and an idle-off servers consumes no energy.
We will consider the normalized energy consumption. 
Thus, the asymptotic steady-state expected normalized energy consumption $\mathbbm{E}[P/340]$ is given by $10/17(q_1+\delta_1)+7/17u = 10/17(1-\delta_0)-3/17u$.
Note that the optimal energy usage (with no wastage, i.e., $\delta_0 = 1-\lambda =0.7,$ $\delta_1 = 0$, $q_1=\lambda =0.3$) is given by $3/17$.
Also recall that the asymptotic expected steady-state waiting time is given by 
$\expt[W] = \lambda^{-1} \sum_{i = 2}^{B} q_i$.\\

In Figure~\ref{fig:increasing-N} average values of the performance metrics, taken over time 0 to 250, have been plotted.
 We can clearly observe that both performance metrics approach the asymptotic values associated with the fixed point of the fluid limit as the number of servers grows large.
Comparison of the results for $\nu = 0.01$ and $\nu = 0.1$ shows that the convergence is substantially faster, and the performance correspondingly closer to the asymptotic lower bound, for shorter setup periods.
This is a manifestation of the fact that, even though the fraction of servers in setup mode vanishes in the limit for any value of $\nu$, the actual fraction for a given finite value of $N$ tends to increase with the mean setup period.
This in turn means that in order for the fluid limit values to be approached within a certain margin, the required value of $N$ increases with the mean setup period, as reflected in Figure~\ref{fig:increasing-N}.
 \\
\begin{figure}
\begin{center}
\includegraphics[width=85mm]{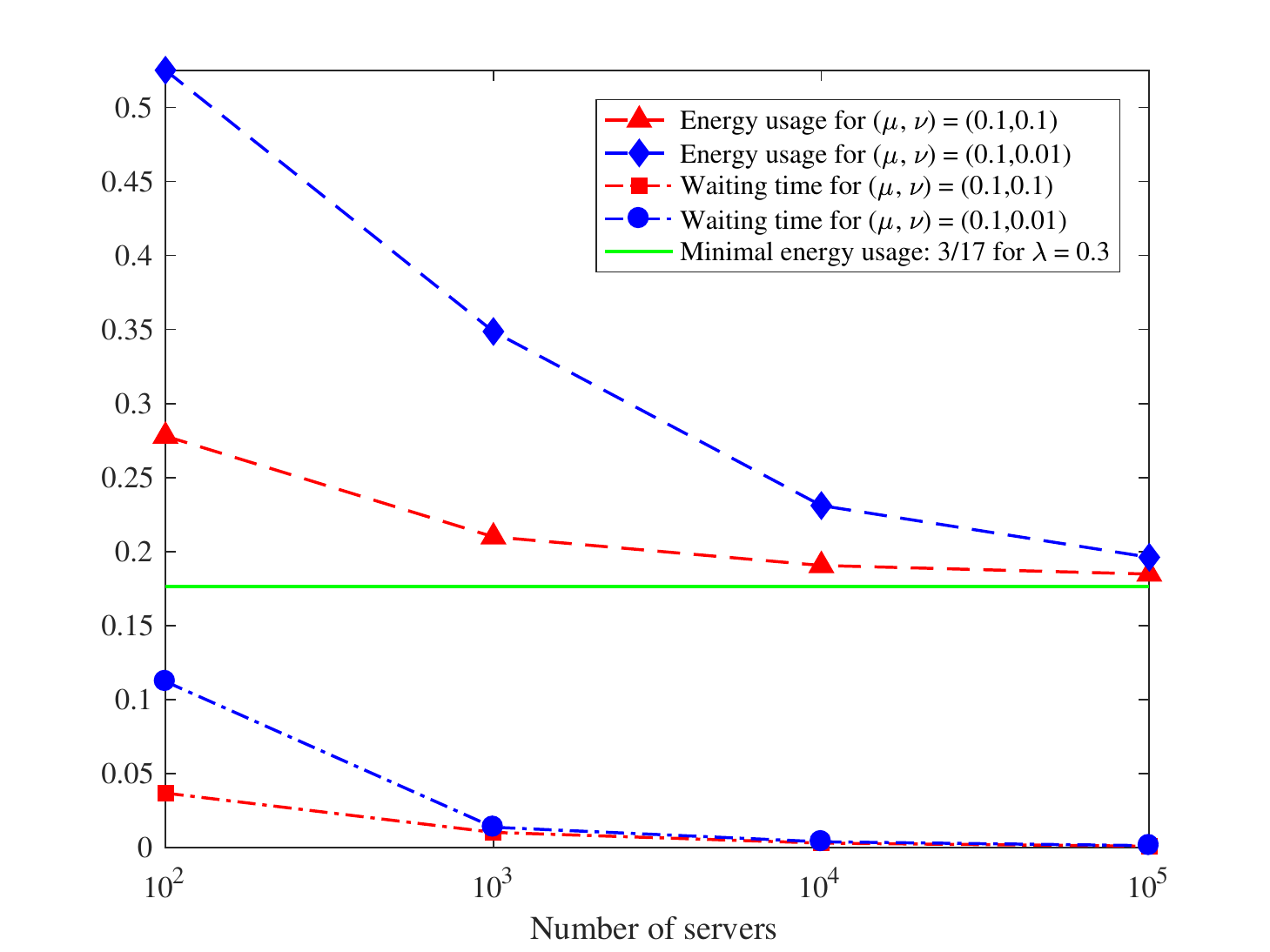}
\end{center}
\caption{Energy usage and mean waiting time for $N=10^2,10^3,10^4,10^5$ servers, mean standby period $\mu^{-1} = 10$, and mean setup periods $\nu^{-1}= 10, 100$.}
\label{fig:increasing-N}
\end{figure}

\begin{figure*}
\begin{center}$
\begin{array}{ccc}
\includegraphics[width=85mm]{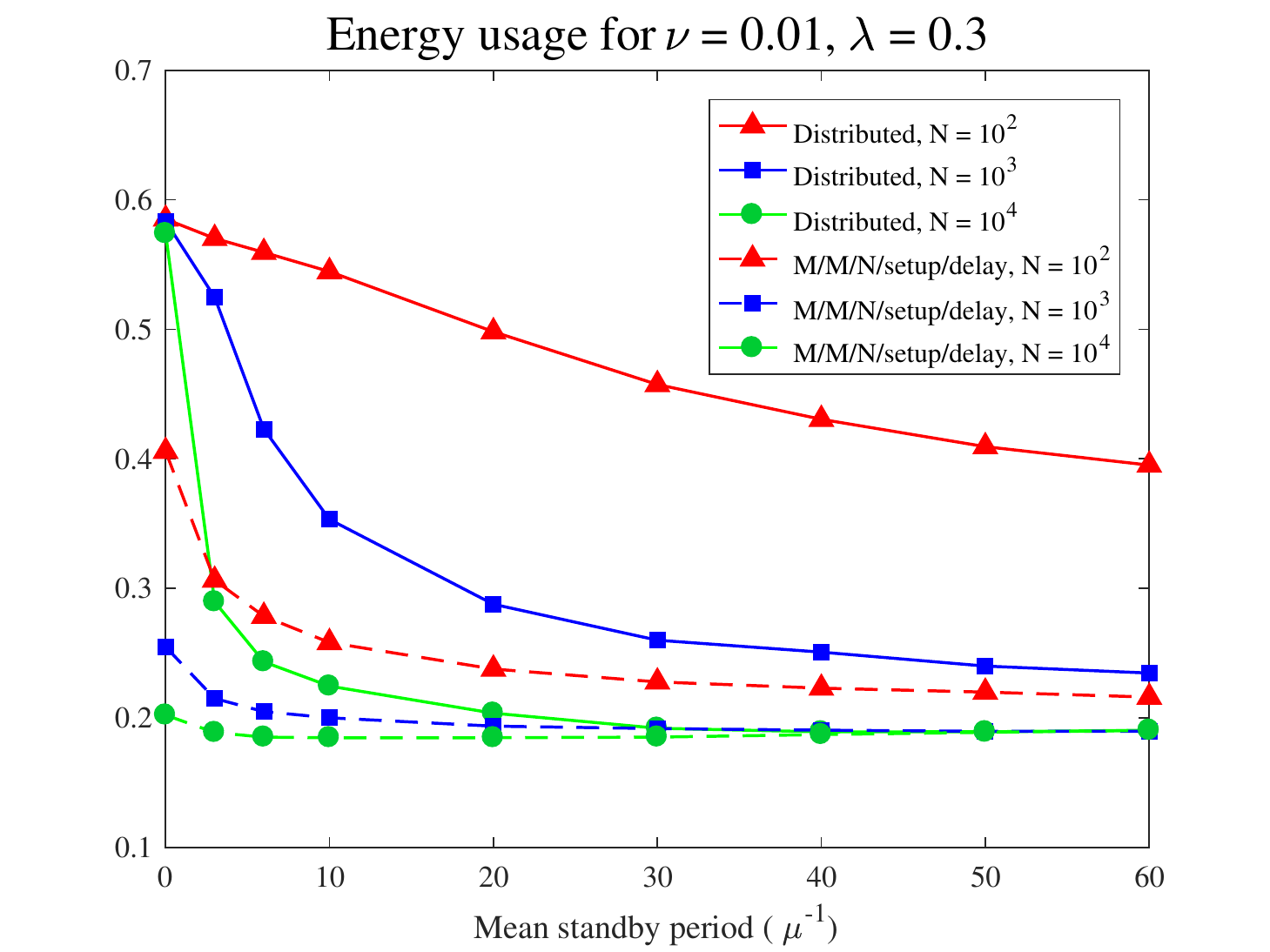}&
\includegraphics[width=85mm]{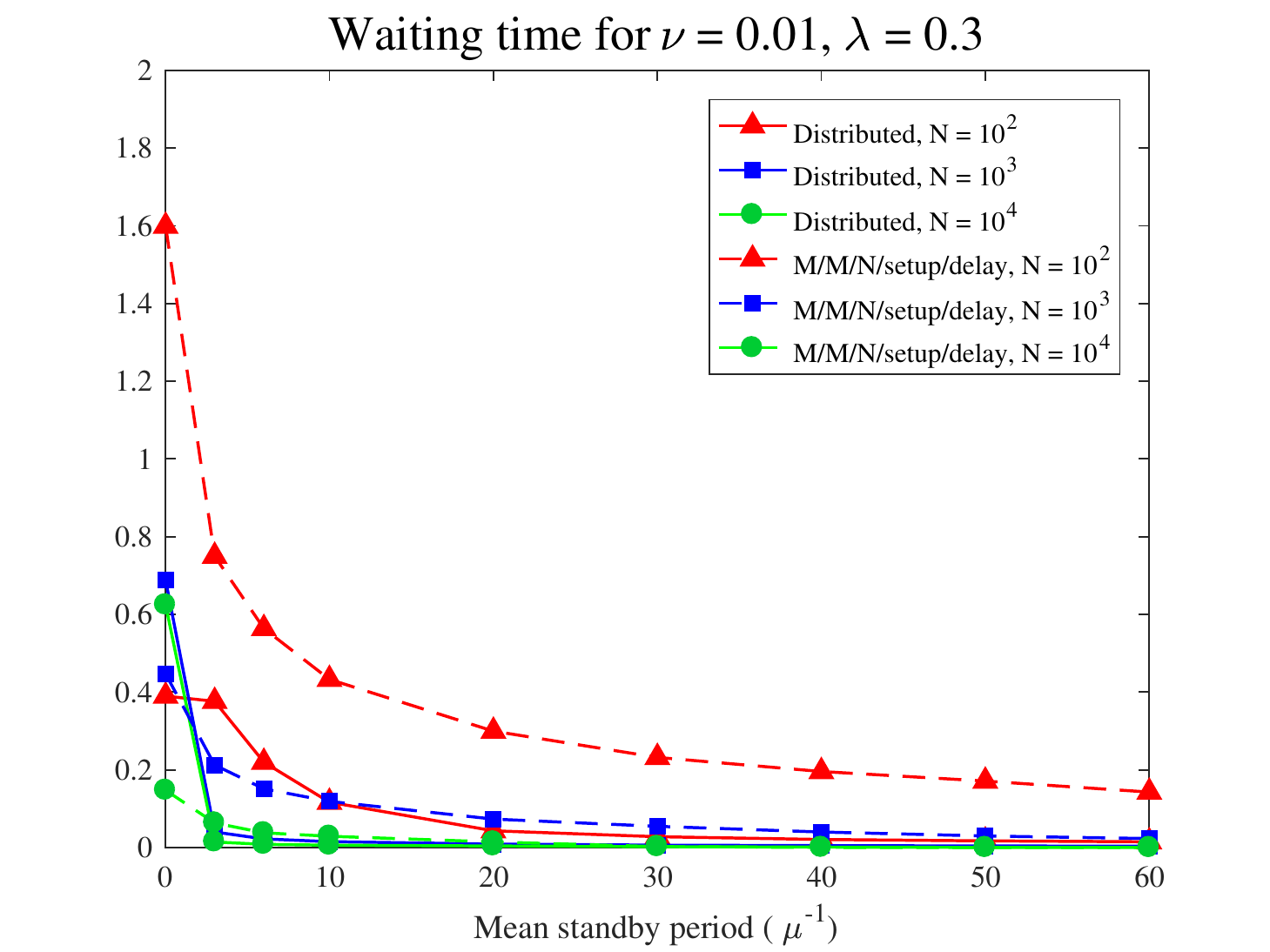}\\
\includegraphics[width=85mm]{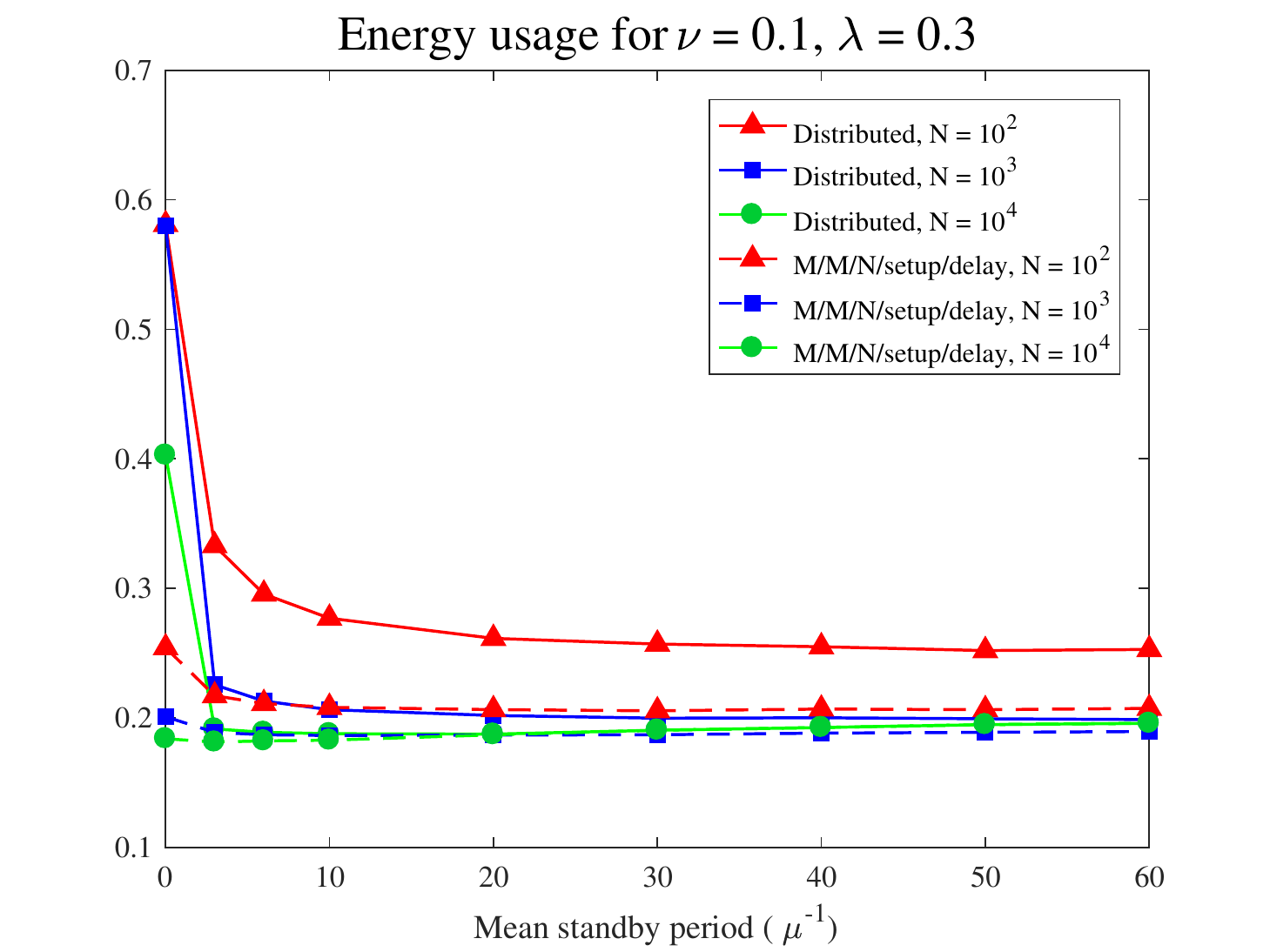}&
\includegraphics[width=85mm]{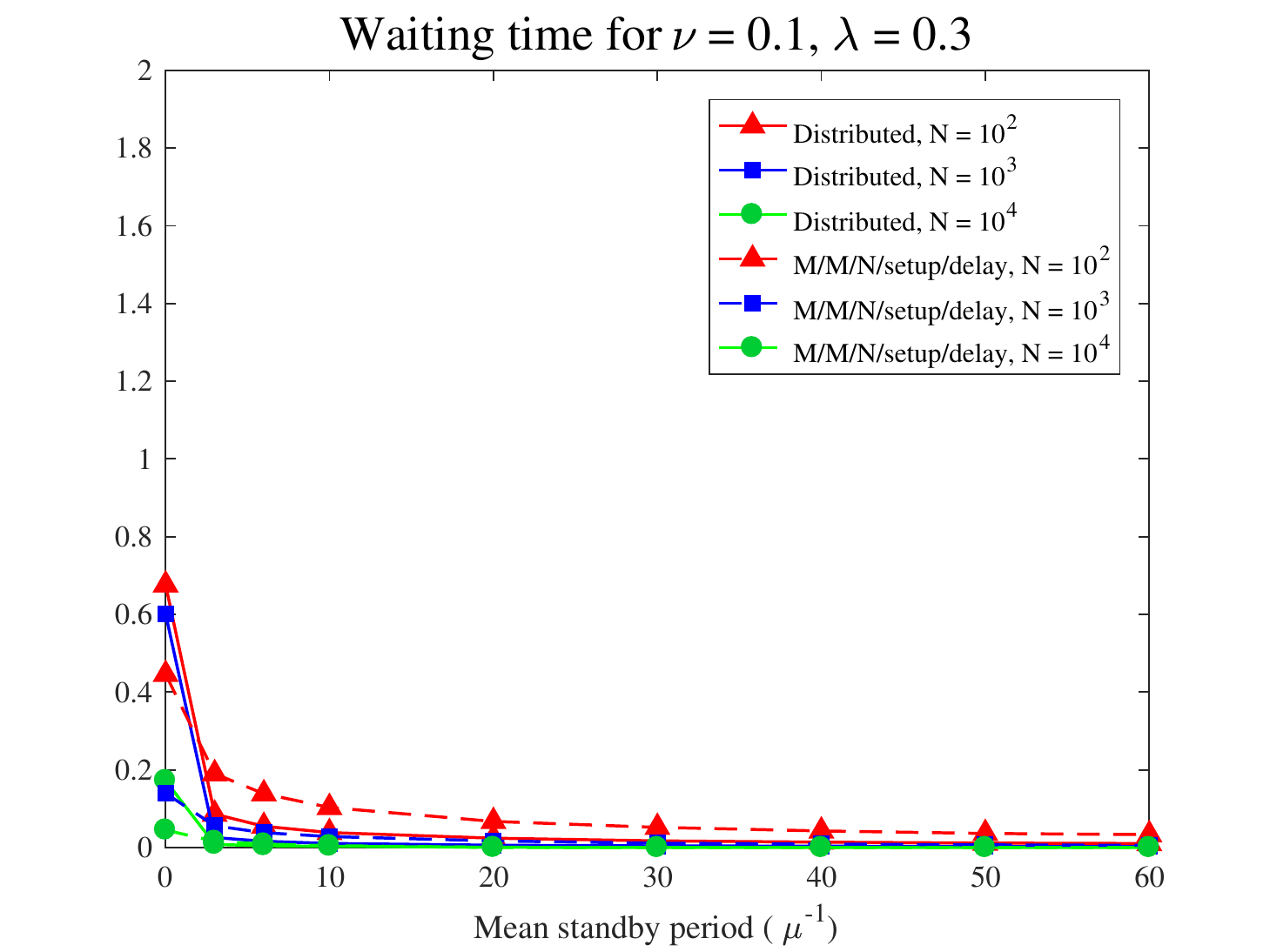}\\
\includegraphics[width=85mm]{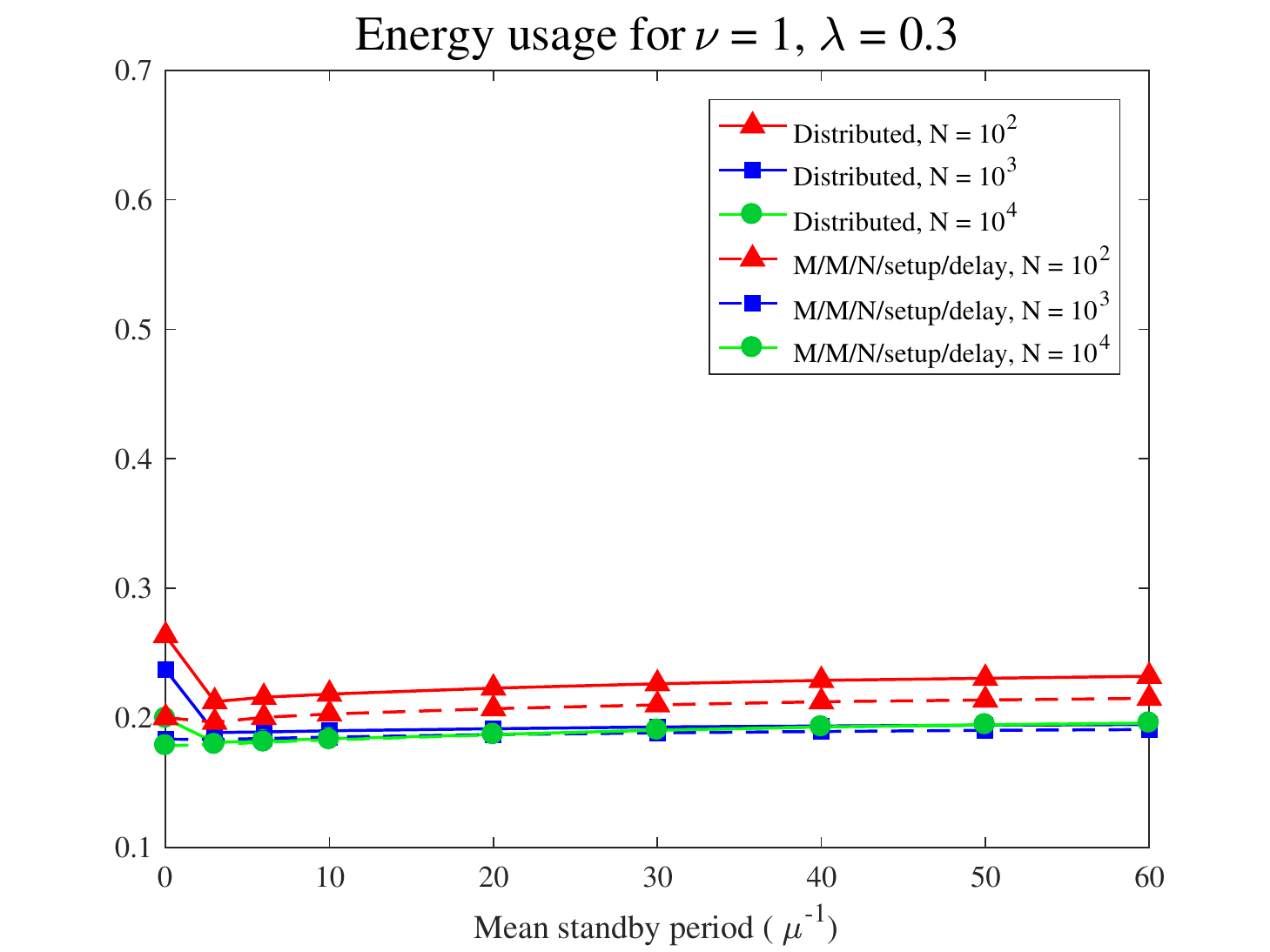}&
\includegraphics[width=85mm]{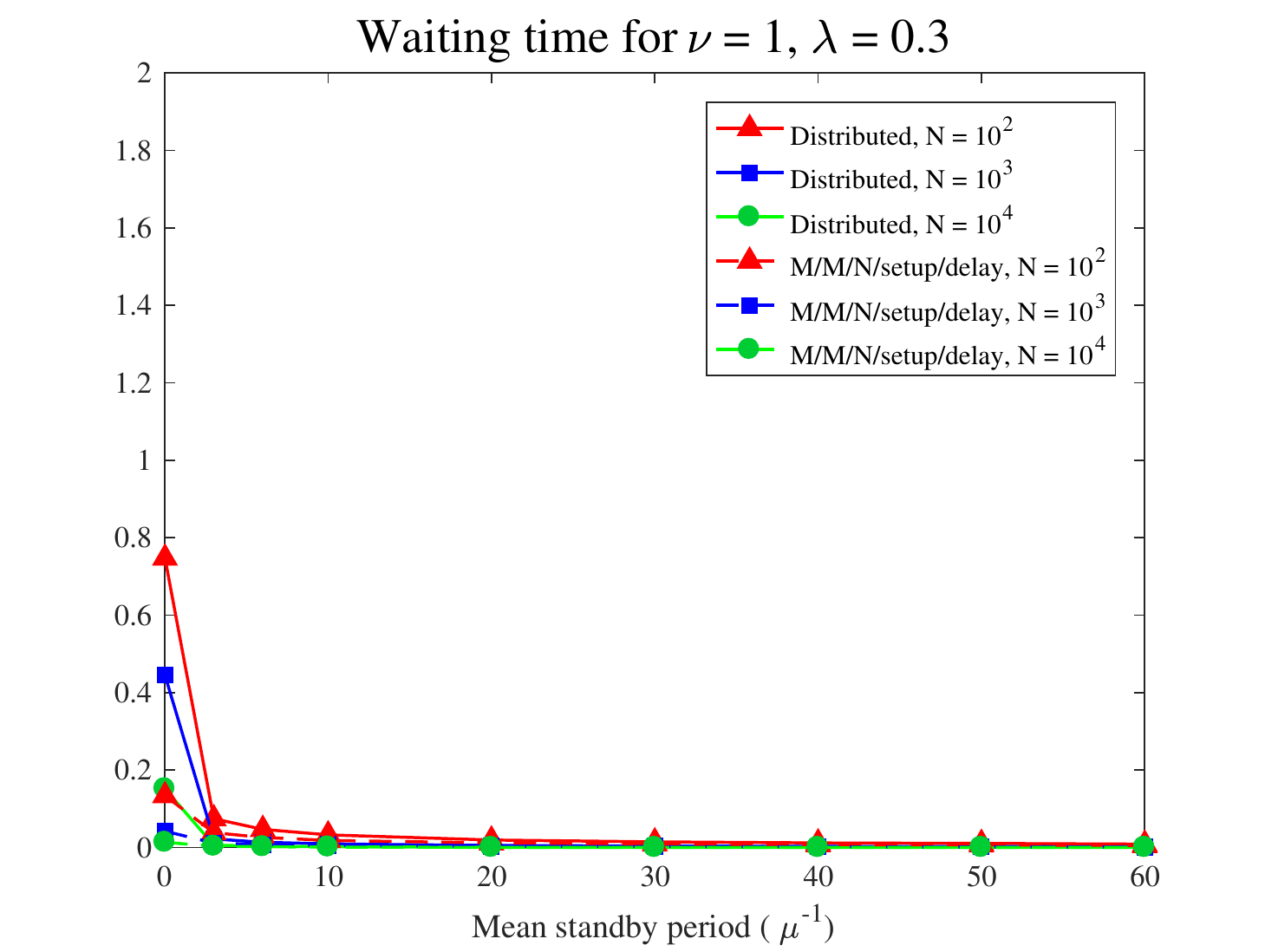}
\end{array}$
\end{center}
\caption{Comparison between TABS and M/M/N/setup/delayedoff schemes as functions  of the mean standby period $\mu^{-1}$ in terms of mean energy consumption and waiting time, for mean setup periods $\nu^{-1} =1, 10, 100$,  $N = 10^2,10^3,10^4$ servers.}
\label{fig:power}
\end{figure*}

In order to further examine the above observations and also investigate the impact of the mean standby period, we present in Figure~\ref{fig:power} the expected waiting time of tasks $\expt[W^N]$ and energy consumption $\expt[P^N]$ for $\lambda = 0.3$ and various values of $N$ and $\nu$, as a function of the mean standby period $\mu^{-1}$.
The results are based on 100 to 200 independent simulation runs, and we confirmed through  careful inspection that the numbers in fact did not show significant variation across runs.
In order to examine the impact of the load, we have also conducted experiments for $\lambda = 0.9$ which are included in the appendix and show qualitatively similar results.
Just like before, the asymptotic values of both performance metrics are clearly approached as the number of servers grows large, and the convergence is noticeably faster and the performance markedly  better, for shorter mean setup periods (larger $\nu$).
The performance impact of the mean standby period $\mu^{-1}$ appears to be somewhat less pronounced.
Both performance metrics generally tend to improve as the mean standby period increases, although the energy consumption starts to slightly rise when the standby period increases above a certain level in scenarios with extremely short setup periods.
The latter observation may be explained as follows. 
For finite $N$-values, if the standby period is extremely small relative to the setup period, then the servers tend to deactivate too often, and as a result, setup procedures are also initiated too often (which in turn involve a relatively long time to become idle-on). 
Note that the servers in setup mode use $P_{\full}$ while providing no service. 
Thus the energy usage decreases by choosing longer standby periods (smaller~$\mu$).
On the other hand, again for small $N$-values, very long standby periods (smaller $\mu$) are not good either.
The reason in this case is straightforward; the idle-on servers will unnecessarily remain idle for a long time, and thus substantially increase energy usage with very little gain in the performance (reduction in waiting time).\\

As mentioned above, the required value of $N$ for the fluid-limit regime to kick in increases with the mean setup period, and broadly speaking, the asymptotic values are approached within a fairly close margin for $N = 10^3$ servers, except when the setup periods are long or the standby periods are extremely short.  
By implication, for scenarios with $N = 10^3$ or more servers, the TABS scheme delivers near-optimal performance in terms of energy consumption and waiting time, provided the setup periods are not too long and the standby periods are not too short.
It is worth observing that setup periods are basically determined by hardware factors and system constraints, while standby periods are design parameters that can be set in a largely arbitrary fashion.  Based on the above observations, a simple practical guideline is to set standby periods to relatively long values.\\

For smaller numbers of servers, long setup periods, or extremely short standby periods, finite-$N$ effects manifest themselves, and the actual performance metrics will differ from the fluid-limit values.
This does not imply though that the performance of the TABS scheme is necessarily far from optimal, since the absolute lower bound attained in the fluid limit may simply not be achievable by any scheme at all for small $N$ values.
\\

\noindent
{\bf Comparison with centralized queue-driven strategies.}
To compare the performance in distributed systems under the TABS scheme with that of the corresponding pooled system under the M/M/N/setup/delayedoff mechanism, 
we also present in Figure~\ref{fig:power} the relevant metrics for the latter scenario. 
Quite surprisingly, even for moderate values of the total number of servers $N$, the performance metrics in a non-work-conserving scenario under the TABS scheme are very close to those for the M/M/N/setup/delayedoff mechanism.
Thus, the TABS scheme provides a significant energy saving in distributed systems which is comparable with that in a work-conserving pooled system, while achieving near zero waiting times as well.
In fact, it is interesting to observe that for relatively long setup periods the waiting time in the distributed system under the TABS scheme is even lower than for the M/M/N/setup/delayedoff mechanism!
This can be understood from the dynamics of the two systems as follows.
When an incoming task does not find an idle server, in both systems an idle-off server $s$ (if available) is switched to the setup mode.
By the time $s$ completes the setup procedure and turns idle-on, in the pooled system if a service completion occurs, then the task is assigned to that new idle-on server and the setup procedure of $s$ is discontinued. 
Therefore, when a next arrival occurs, the setup procedure must be initiated again.
As a result, this might cause the effective average waiting time to become higher.
On the other hand, in the distributed system once a setup procedure is initiated, it is completed in any event. 
This explains why for relatively long setup periods the TABS scheme provides a lower waiting time than the M/M/N/setup/delayedoff mechanism.

\section{Proofs}\label{sec:proofs}
The proof of Theorem~\ref{th: fluid} consists of describing the evolution of the system as a suitable time-changed Poisson process, which can be further decomposed into a martingale part and a drift part.
This formulation can be viewed as a \emph{density-dependent population process} (cf.~\cite[Chapter 11]{EK2009}).
 The martingale fluctuations become negligible on the fluid scale, and the drift terms converge to deterministic limits.
While the convergence of the martingale fluctuations is fairly straightforward to show, the analysis of the drift term is rather involved since the derivative of the drift is not continuous. 
As a result, the classical approaches developed by Kurtz~\cite{EK2009} cannot be applied in the current scenario.
In the literature, these situations have been tackled in various different ways~\cite{HK94, PW13, K92, GG12, TX11, GG10, B16, BG16}.
In particular, we leverage the time-scale separation techniques developed in~\cite{HK94} in order to identify the limits of drift terms.\\

Let us introduce the variables $U^N(t) = N-Q_1^N(t)-\Delta_0^N(t)-\Delta_1^N(t)$, $u^N(t) = U^N(t)/N$,
$I_0^N(t) = \ind{U^N(t)>0}$, and $I_1^N(t) = \ind{\Delta_0^N(t)>0}$.
Note that $U^N(t)$ represents the number of idle-on servers at time $t$.\\

\noindent
{\bf Random measure representation.}
We will now write the system evolution equation in terms of a suitable random measure.
The transition rates of the process $\{\ZZ^N(t)\}_{t\geq 0}:=\{(\Delta_0^N(t),U^N(t))\}_{t\geq 0}$ are described as follows.
\begin{enumerate}[(i)]
 \item When an idle server turns-off, $\Delta_0^N$ increases by one and $U^N$ decreases by one, and this occurs at rate $N\mu U^N$;
 \item When a server is requested to initiate the setup procedure, $U^N$ must be zero at that epoch. Thus,  $\Delta_0^N$ decreases by one while $U^N$ remains unchanged, and this occurs at rate $\lambda_N(t)\ind{U^N = 0 , \Delta_0^N> 0}$;
 \item When a busy server becomes idle, or a server finishes its setup procedure to become idle-on,  $\Delta_0^N$ remains unchanged while $U^N$ increases by one, and this occurs at rate $N(q_1^N-q_2^N+\nu\delta_1^N)$;
 \item When an arriving task is assigned to an idle server, $\Delta_0^N$ remains unchanged while $U^N$ decreases by one, and this occurs at rate $\lambda_N(t)\ind{U^N>0}$.
\end{enumerate}

    Let $\bar{\mathbbm{Z}}_+=\mathbbm{Z}_+\cup\{\infty\}$ denote the one-point compactification of the set of non-negative integers, equipped with the Euclidean metric, and the Borel $\sigma$-algebra $\mathfrak{B}$, induced by the mapping $f:\bar{\mathbbm{Z}}_+\to[0,1]$ given by $f(x)=1/(x+1)$. 
    Let $\mathbf{v}^N(t)$ denote the vector $(\mathbf{q}^N(t),\dd^N(t))$.
  
  Observe that $\{(\mathbf{v}^N(t),\ZZ^N(t))\}_{t\geq 0}$ is a  Markov process defined on $E\times\bar{\mathbbm{Z}}_+^2$. 
  Further, equip $[0,\infty)$ with the usual Euclidean metric and the Borel $\sigma$-algebra $\mathfrak{T}$.
We define a random measure $\alpha^N$ on the product space $[0,\infty)\times\bar{\mathbbm{Z}}_+^2$ by
\begin{equation}
\alpha^N(A_1\times A_2):=\int_{A_1} \ind{\ZZ^N(s)\in A_2}\dif s,
\end{equation}
for $A_1\in \mathfrak{T}$, $A_2\in\mathfrak{B}$. 
Define
\begin{align*}
\mathcal{R}_1 &= \{(z_1,z_2)\in \Z_+^2: z_2 = 0\},\\ \mathcal{R}_2 &= \{(z_1,z_2)\in \Z_+^2: z_2 = 0, z_1>0\}.
\end{align*}
Then the Markov process $\big\{(\qq^N(t),\dd^N(t))\big\}_{t\geq 0}$ can be written as in~\eqref{eq:fluid mart2}, where $\mathcal{M}_{A}$, $\mathcal{M}_{0}$, $\mathcal{M}_{1}$, $\mathcal{M}_{i,D}$ for $i = 1,\dots,B$ are square-integrable martingales.
 \begin{table*}
\begin{eq}\label{eq:fluid mart2}
q^N_1(t)&=q^N_1(0)+\frac{1}{N}\left(\mathcal{M}_{  A} (t) - \mathcal{M}_{ 1, D}(t)\right) + \int_{[0,t]\times\mathcal{R}_1^c} \lambda(s)\dif \alpha^N-\int_0^t(q^N_1(s)-q^N_2(s))\dif s,\\
q^N_i(t)&=q^N_i(0)+\frac{1}{N}\left(\mathcal{M}_{ A}(t)-\mathcal{M}_{ i, D}(t)\right)+\int_{[0,t]\times\mathcal{R}_1} \frac{q^N_{i-1}(s)-q^N_i(s)}{q^N_1(s)}\lambda(s)\dif \alpha^N-\int_0^t (q^N_i(s)-q^N_{i+1}(s))\dif s, \quad i = 2,\dots,B,\\
\delta_0^N(t)&=\delta_0^N(0)+\frac{1}{N}\left(\mathcal{M}_{0}(t) - \mathcal{M}_{ A}(t) \right) +\mu\int_0^t u^N(s)\dif s -\int_{[0,t]\times\mathcal{R}_2}\lambda(s)\dif \alpha^N,\\
\delta_1^N(t) &= \delta_1^N(0)+\frac{1}{N}\left(\mathcal{M}_{  A}(t)- \mathcal{M}_{1}(t)\right)+\int_{[0,t]\times\mathcal{R}_2}\lambda(s)\dif \alpha^N-\nu\int_0^t\delta_1^N(s)\dif s.
\end{eq}
\end{table*}
A step-by-step derivation of the representation in~\eqref{eq:fluid mart2} is presented in the appendix.
We first show that the scaled martingale parts converge in probability to zero processes as $N\to\infty$.
\begin{proposition}
\label{prop:mart zero1}
For any $T\geq 0$, $\sup_{t\in [0,T]}|\mathcal{M}_{k}(t)|/N\pto 0$ for $k = A,0,1$, and $\sup_{t\in [0,T]}|\mathcal{M}_{ i,D}(t)|/N\pto 0$, for all $i = 1,\dots B$.
\end{proposition}

Let $\mathfrak{L}$ denote the space of all measures $\gamma$ on $[0,\infty)\times \bar{\mathbbm{Z}}_+^2$ satisfying $\gamma([0,t]\times\bar{\mathbbm{Z}}_+^2)= t$, endowed with the topology corresponding to weak convergence of measures restricted to $[0,t]\times \bar{\mathbbm{Z}}_+^2$ for each $t$.  We have the following lemma:
\begin{lemma}[Relative compactness]
\label{lem:rel compactness}
Suppose that $\mathbf{v}^N(0)$ converges weakly to $\mathbf{v}^\infty = (\qq^\infty,\dd^\infty)\in E$ as $N\to\infty$, with $q_1^\infty>0$. Then the sequence of processes $\{(\mathbf{v}^N(\cdot),\alpha^N)\}_{N\geq 1}$ is relatively compact in $D_{E}[0,\infty)\times\mathfrak{L}$ and the limit $(\mathbf{v}(\cdot),\alpha)$ of any convergent subsequence satisfies
\begin{eq}\label{eq:rel compact}
q_1(t)&=q_1^\infty+ \int_{[0,t]\times\mathcal{R}_1^c} \lambda(s)\dif \alpha-\int_0^t(q_1(s)-q_2(s))\dif s\\
q_i(t)&=q_i^\infty+\int_{[0,t]\times\mathcal{R}_1} \frac{q_{i-1}(s)-q_i(s)}{q_1(s)}\lambda(s)\dif \alpha\\
&\hspace{2cm} -\int_0^t (q_i(s)-q_{i+1}(s))\dif s, \quad i = 2,\dots,B,\\
\delta_0(t)&=\delta_0^\infty+\mu\int_0^t u(s)\dif s -\int_{[0,t]\times\mathcal{R}_2}\lambda(s)\dif \alpha\\
\delta_1(t) &= \delta_1^\infty+\int_{[0,t]\times\mathcal{R}_2}\lambda(s)\dif \alpha-\nu\int_0^t\delta_1(s)\dif s,
\end{eq}
with $u(t) = 1-q_1(t)-\delta_0(t)-\delta_1(t).$
\end{lemma}
To prove Lemma~\ref{lem:rel compactness}, we verify the conditions of relative compactness from~\cite[Corollary~3.7.4]{EK2009}. 
We present the precise conditions and the proof of Lemma~\ref{lem:rel compactness} in the appendix.  
We will now prove the fluid-limit result stated in Theorem~\ref{th: fluid}.

\begin{proof}[{Proof of Theorem~\ref{th: fluid}}]
Using \cite[Theorem 3]{HK94}, we can conclude that the measure $\alpha$ can be represented as
\begin{equation}
\alpha (A_1\times A_2) = \int_{A_1} \pi_{\qq(s),\dd(s)}(A_2)\dif s,
\end{equation}
for measurable subsets $A_1\subset[0,\infty)$, and $A_2\subset \bar{\Z}_+^2$, where for any $(\bld{q},\dd) \in E$, $\pi_{\bld{q},\dd}$ is given by some  stationary distribution of the Markov process with transitions  
\begin{eq}\label{eq:stationary-bd-chain}
 (Z_1,Z_2)\rightarrow 
\begin{cases}
(Z_1,Z_2)+(1,-1)& \mbox{at rate } \mu u\\
(Z_1,Z_2)+(-1,0)& \mbox{at rate } \lambda\ind{Z_2=0,Z_1>0}\\
(Z_1,Z_2)+(0,1)& \mbox{at rate } q_1-q_2+\nu\delta_1\\
(Z_1,Z_2)+(0,-1)& \mbox{at rate } \lambda \ind{Z_2>0},
\end{cases}
\end{eq}
with $u=1-q_1-\delta_0-\delta_1$.
Additionally, the measure $\pi_{\qq,\dd}$ satisfies $\pi_{\qq,\dd}(Z_2 = \infty) = 1$, if $u>0$ and $\pi_{\qq,\dd}(Z_1 = \infty) = 1$ if $\delta_0>0$.
Thus we will show that for any $(\qq,\dd)\in E$, $\pi_{\bld{q},\dd}$ is unique, and that $\pi_{\qq(s),\dd(s)}(\mathcal{R}_1)=p_{0}(\qq(s),\dd(s))$ and $\pi_{\qq(s),\dd(s)}(\mathcal{R}_2)=(1-p_{0}(\qq(s),\dd(s)))\ind{\delta_0(s)>0}$ as described in Theorem~\ref{th: fluid} (we have omitted the argument $\lambda(s)$ in $p(\cdot,\cdot,\cdot)$ to avoid cumbersome notation). 
We will verify the uniqueness of the stationary measure $\pi_{\qq,\dd}$ of the Markov process $(Z_1,Z_2)$ subsequently case-by-case.\\

\noindent
{\bf Case-I: $\bld{u>0,$ $\delta_0>0}$.}
In this case, by the definition of $\pi_{\qq,\dd}$ stated above, $\pi_{\qq,\dd} (Z_2 = Z_1 =\infty) = 1$. Thus, $\pi_{\qq,\dd}(\mathcal{R}_1) = \pi_{\qq,\dd}(\mathcal{R}_2) = 0$.\\

\noindent
{\bf Case-II: $\bld{u>0,$ $\delta_0=0}$.} 
Here by definition of $\pi_{\qq,\dd}$  $\pi_{\qq,\dd} (Z_2  =\infty) = 1$. 
However, if $Z_2 = \infty$, then by~\eqref{eq:stationary-bd-chain},  $Z_1$ increases by one at rate $\mu u$, and decreases at rate 0. Since $\pi_{\qq,\dd}$ is the stationary measure, we also have $\pi_{\qq,\dd} (Z_1  =\infty) = 1$, and thus,  $\pi_{\qq,\dd}(\mathcal{R}_1) = \pi_{\qq,\dd}(\mathcal{R}_2) = 0$.\\

\noindent
{\bf Case-III: $\bld{u=0,$ $\delta_0>0}$.}
 In this case, $\pi_{\qq,\dd} (Z_1  =\infty) = 1$. 
Again note that if $Z_1 = \infty$, then by~\eqref{eq:stationary-bd-chain}, $Z_2$ increases by one at rate $q_1-q_2+\nu\delta_1$, and decreases by one at rate $\lambda\ind{Z_2>0}$. Thus, 
\begin{itemize}
\item if $q_1-q_2+\nu\delta_1\geq \lambda$, then $\pi_{\qq,\dd} (Z_2  =0) = 0$, and consequently, $\pi_{\qq,\dd}(\mathcal{R}_1) = \pi_{\qq,\dd}(\mathcal{R}_2) = 0$,
\item if $q_1-q_2+\nu\delta_1< \lambda$,  then $\pi_{\qq,\dd} (Z_2  = 0) = \lambda^{-1}(q_1-q_2+\nu\delta_1)$, and  $\pi_{\qq,\dd}(\mathcal{R}_1) = \pi_{\qq,\dd}(\mathcal{R}_2) = \lambda^{-1}(q_1-q_2+\nu\delta_1)$.
\end{itemize}

\noindent
{\bf Case-IV: $\bld{u=0,$ $\delta_0=0}$.} 
Observe that in this case, due to physical constraints, it must be that $\pi_{\qq,\dd}(\mathcal{R}_2) = 0$. 
To see this, recall the evolution equation from \eqref{eq:rel compact}. 
Note that $\delta_0(t) = 0$ forces its derivative to be non-negative (since $\delta_0$ is non-negative), and thus $\delta_0'(t)\geq 0$.
Now, $\pi_{\qq(t),\dd(t)}(\mathcal{R}_2) > 0$ implies that $\delta_0'(t)<0$, and  hence, this leads to a contradiction.
Furthermore, $\pi_{\qq,\dd}(Z_2=0,Z_1>0) = 0$ implies that $\pi_{\qq,\dd}(Z_2=0) = \pi_{\qq,\dd}(Z_2=0,Z_1=0)$.  
Again, if $Z_1 = 0$, then by~\eqref{eq:stationary-bd-chain}, $Z_2$ increases by one at rate $q_1-q_2+\nu\delta_1$, and decreases by one at rate $\lambda\ind{Z_2>0}$. 
Thus, an argument similar to Case-III yields that $\pi_{\qq,\dd}(\mathcal{R}_1)  = 0$, if $q_1-q_2+\nu\delta_1\geq \lambda$, and $\pi_{\qq,\dd}(\mathcal{R}_1) = \lambda^{-1}(q_1-q_2+\nu\delta_1)$, if $q_1-q_2+\nu\delta_1<\lambda$.
Combining Cases I-IV, we have 
\begin{align*}
 \pi_{\qq,\dd}(\mathcal{R}_1) &= 1- p_0(\qq,\dd,\lambda), \quad \pi_{\qq,\dd}(\mathcal{R}_2) = \ind{\delta_0>0}\pi_{\qq,\dd}(\mathcal{R}_1),
\end{align*}
and the proof of Theorem~\ref{th: fluid} follows from Lemma~\ref{lem:rel compactness}.
\end{proof}

\begin{proof}[Proof sketch of Proposition~\ref{prop:glob-stab}]
We now provide a brief proof outline of Proposition~\ref{prop:glob-stab}.
A detailed proof is presented in the appendix.\\

\noindent
{\bf Convergence of $\mathbf{q_1(t)}$.}
First we will establish that
$q_1(t)\to \lambda$ as $t\to\infty$.
The high-level intuition behind the proof can be described in two steps as follows.

(1) First we prove that $\liminf_{t\to\infty}q_1(t)\geq \lambda$.
Assume the contrary.
Because $q_1(t)$ can be shown to be non-decreasing when $q_1(t)\leq \lambda$, there must exist an $\varepsilon>0$, such that 
\begin{equation}\label{eq:contradict1}
q_1(t)\leq \lambda-\varepsilon\nu, \quad \forall\ t\geq 0.
\end{equation}
If $q_1(t)$ were to remain below $\lambda$ by a non-vanishing margin, then the (scaled) rate $q_1(t) - q_2(t)$ of busy servers turning idle-on would not be high enough to match the (scaled) rate $\lambda$ of incoming jobs.
If there are idle-on servers or sufficiently many servers in setup mode, we can still assign incoming jobs to idle-on servers, but this drives up the fraction of busy servers $q_1(t)$ and cannot continue indefinitely due to \eqref{eq:contradict1}.
This means that we cannot initiate an unbounded number of setup procedures.
Since we cannot continue to have idle-on servers either, this also implies that a non-vanishing fraction of the jobs cannot be assigned to idle servers, and hence we will initiate an unbounded number of setup procedures, hence contradiction. 

(2) Next we show that $\limsup_{t\to\infty} q_1(t)\leq \lambda$.
Suppose not, i.e., $\limsup_{t\to\infty}q_1(t) = \lambda+\varepsilon$ for some $\varepsilon>0$.
Recall that $q_1(t)$ is non-decreasing when $q_1(t)\leq \lambda$.
Hence, there must exist a $t_0$ such that $q_1(t)\geq \lambda$ $\forall\ t\geq t_0$.
If $q_1(t)$ were to get above $\lambda$ by a non-vanishing margin infinitely often, then the cumulative number of departures would exceed the cumulative number of arrivals by an infinite amount, which cannot occur since the (scaled) initial number of tasks is bounded.\\

\noindent
{\bf Convergence of $\mathbf{q_2(t)}$.}
Based on the fact that $q_1(t)\to \lambda$ as $t\to\infty$, we now claim that
$q_2(t)\to 0$ as $t\to\infty$.
The high-level idea behind the claim is as follows. 
From the convergence of $q_1(t)$, we know that after a large enough time, $q_1(t)$ will always belong to a very small neighborhood of $\lambda$.
On the other hand, if $q_2(t)$ does not converge to 0, then it must  have a strictly positive limit point. 
In that case, since the rate of decrease of $q_2(t)$ is at most $q_2(t)$, it will be bounded away from 0 for a fixed amount of time infinitely often.
In the meantime, the rate at which busy servers become idle-on will be strictly less than the arrival rate of tasks.
This in turn, will cause $q_1(t)$ to increase substantially compared to the small neighborhood where it is supposed to lie, which leads to a contradiction. \\

\noindent
{\bf Convergence of $\boldsymbol{\delta_0(t)}$ and $\boldsymbol{\delta_1(t)}$.}
Since $q_1(t) - q_2(t)\to \lambda$ and $q_2(t)\to 0$, as $t\to\infty$,
it follows that $p_0(\mathbf{q}(t),\boldsymbol{\delta}(t),\lambda)\to 1$ as $t\to\infty.$
From the evolution equation of $\delta_0(t)$, the rate of increase goes to zero, and since the rate of decrease is proportional to $\delta_0(t)$, using Gronwall's inequality, we obtain $\delta_1(t)\to 0$ as $t\to\infty$.
Consequently, $\delta_0(t)\to 1-\lambda$ as $t\to\infty$. This completes the proof of Proposition~\ref{prop:glob-stab}.
\end{proof}

\section{Conclusions}\label{sec:conclusion}
Centralized queue-driven auto-scaling techniques do not cover
scenarios where load balancing algorithms immediately distribute
incoming tasks among parallel queues, as typically encountered
in large-scale data centers and cloud networks.
Motivated by these observations, we proposed a joint auto-scaling
and load balancing scheme, which does not require any global queue
length information or explicit knowledge of system parameters.
Fluid-limit results for a large-capacity regime show that the
proposed scheme achieves asymptotic optimality in terms of response
time performance as well as energy consumption.
At the same time, the proposed scheme operates in a distributed
fashion, and involves only a constant communication overhead per task,
ensuring scalability to massive numbers of servers.
This demonstrates that, rather remarkably, ideal response time
performance and minimal energy consumption can be simultaneously
achieved in large-scale distributed systems.

Extensive simulation experiments support the fluid-limit results,
and reveal only a slight trade-off between the mean waiting time
and energy wastage in finite-size systems.
In particular, we observe that suitably long but finite standby
periods yield near-minimal waiting time and energy consumption,
across a wide range of setup durations.
We expect that a non-trivial trade-off between response time
performance and (normalized) energy consumption arises at the
diffusion level, and exploring that conjecture would be
an interesting topic for further research.
It might be worth noting that in the present paper, we have not taken the communication delay into consideration, and assumed that the message transfer is instantaneous.
This is a reasonable assumption when the communication delay is insignificant relative to the typical duration of the service period of a job.
When the communication delay is non-negligible, one might modify the TABS scheme where a task is discarded if it happens to land on an idle-off server.
In this modified scheme, the asymptotic fraction of lost tasks in steady state should be negligible, since the rate at which idle-on servers are turning of is precisely zero at the fixed point, and it would be useful to further examine the impact of communication delays.

\section{Acknowledgments}
This research was financially supported by The Netherlands Organization for Scientific Research (NWO) through Gravitation Networks grant -- 024.002.003 and TOP-GO grant -- 613.001.012.

\bibliographystyle{ACM-Reference-Format}
\bibliography{bib-auto}

\appendix

\section{Fluid convergence}\label{app:conv}
 First, we verify the existence of the coefficients $p_i(\cdot,\cdot,\cdot)$ for all $t\geq 0$, $i = 1,2,\ldots,B$.
From the assumptions of Theorem~\ref{th: fluid}, and the fact that $\lambda(t)$ is bounded away from 0 (by some $\lambda_{\min}$ say), we claim that if $q_1(0)=q_1^\infty>0$, then $q_1(t)>0$ for all $t\geq 0$. 
To see this, it is enough to observe that in the fluid limit the rate of change of $q_1(t)$ is non-negative whenever $q_1(t)<\lambda_{\min}$.
Indeed, if $q_1(t)<\lambda_{\min}$, then 
\begin{align*}
&\lambda(t)p_0(\qq(t),\dd(t),\lambda(t))-(q_1(t)-q_2(t))\\
&\geq \min\{\lambda(t) - (q_1(t)-q_2(t)), \delta_1(t)\nu\}\\
&\geq \min\{\lambda_{\min}-q_1(t),\delta_1\nu\}\geq 0,
\end{align*}
and thus the claim follows.
Therefore below we will prove Theorem~\ref{th: fluid} until the time $q_1^N$ hits 0, and the above argument then shows that
if $q_1^N(0)\pto q_1^\infty>0$, then on any finite time interval $[0,T]$, with probability tending to 1, the process $q_1^N(\cdot)$ is bounded away from 0, proving the theorem for any finite time interval.
  \\

\noindent
{\bf Martingale representation.}
For a unit-rate Poisson process $\big\{\mathcal{N}(t)\big\}_{t\geq 0}$ and a real-valued c\`adl\`ag process $\{A(t)\}_{t\geq 0}$, the random time-change~\cite{PTRW07, EK2009} $\big\{\mathcal{N}(\int_0^tA(s)\dif s)\big\}_{t\geq 0}$ is the unique process such that
\begin{equation}\label{def:timechange}
 \mathcal{N}\bigg(\int_0^tA(s)\dif s\bigg) - \int_0^tA(s)\dif s \quad\text{ is a martingale.}
\end{equation} 
Thus the evolution of the system is described by \eqref{eq:poisson-descr}, where $\mathcal{N}_{ A}$, $\mathcal{N}_{ i,D}$ for $i= 1,\dots,B$, $\mathcal{N}_{0}$, $\mathcal{N}_{1}$ are independent unit-rate Poisson processes.
\begin{table*}
\begin{eq}\label{eq:poisson-descr}
Q^N_1(t)&=Q^N_1(0)+\mathcal{N}_{ A}\left(\int_0^t (1-I_0^N(s))\lambda_N(s)\dif s\right)-\mathcal{N}_{ 1, D}\left(\int_0^t(Q^N_1(s)-Q^N_2(s))\dif s\right),\\
Q^N_i(t)&=Q^N_i(0)+\mathcal{N}_{  A}\left(\int_0^t  I_0^N(s)\frac{Q^N_{i-1}(s)-Q^N_i(s)}{Q^N_1(s)}\lambda_N(s)\dif s\right)-\mathcal{N}_{ i, D}\left(\int_0^t (Q^N_i(s)-Q^N_{i+1}(s))\dif s\right), \quad i = 2,\dots,B,\\
\Delta_0^N(t)&=\Delta_0^N(0)+\mathcal{N}_{0}\left(\mu\int_0^t U^N(s)\dif s\right)-\mathcal{N}_{ A}\left(\int_0^tI_0^N(s)I_1^N(s)\lambda_N(s)\dif s\right),\\
\Delta_1^N(t) &= \Delta_1^N(0)+\mathcal{N}_{ A}\left(\int_0^t I_0^N(s)I_1^N(s)\lambda_N(s)\dif s\right)-\mathcal{N}_{1}\left(\nu\int_0^t\Delta_1^N(s)\dif s\right),
\end{eq}
\end{table*}
Using~\eqref{def:timechange} and \eqref{eq:poisson-descr}, we obtain the martingale representation of the process as in~\eqref{eq:mart},
where recall that $\mathcal{M}_{A}$, $\mathcal{M}_{0}$, $\mathcal{M}_{1}$, $\mathcal{M}_{i,D}$ for $i = 1,\dots,B$ are square-integrable martingales. 
\begin{table*}
\begin{eq}\label{eq:mart}
Q^N_1(t)&=Q^N_1(0)+\mathcal{M}_{  A} (t) - \mathcal{M}_{ 1, D}(t) + \int_0^t (1-I_0^N(s))\lambda_N(s)\dif s-\int_0^t(Q^N_1(s)-Q^N_2(s))\dif s,\\
Q^N_i(t)&=Q^N_i(0)+\mathcal{M}_{ A}(t)-\mathcal{M}_{ i, D}(t)+\int_0^t I_0^N(s)\frac{Q^N_{i-1}(s)-Q^N_i(s)}{Q^N_1(s)}\lambda_N(s)\dif s-\int_0^t (Q^N_i(s)-Q^N_{i+1}(s))\dif s, \quad i = 2,\dots,B,\\
\Delta_0^N(t)&=\Delta_0^N(0)+\mathcal{M}_{0}(t) - \mathcal{M}_{ A}(t) +\mu\int_0^t U^N(s)\dif s -\int_0^tI_0^N(s)I_1^N(s)\lambda_N(s)\dif s,\\
\Delta_1^N(t) &= \Delta_1^N(0)+\mathcal{M}_{  A}(t)- \mathcal{M}_{1}(t)+\int_0^t I_0^N(s)I_1^N(s)\lambda_N(s)\dif s-\nu\int_0^t\Delta_1^N(s)\dif s,
\end{eq}
\end{table*}
The fluid-scaled martingale decomposition is thus given by~\eqref{eq:fluid mart}.
Note that the process $\{\ZZ^N(t)\}_{t\geq 0}$ defined in Section~\ref{sec:proofs} determines the system constraints (indicator terms $I_0^N$ and $I_1^N$) in \eqref{eq:fluid mart}. 
Thus,~\eqref{eq:fluid mart} can be written in terms of the random measure $\alpha^N$ as in~\eqref{eq:fluid mart2}.

\begin{table*}
\begin{eq}\label{eq:fluid mart}
q^N_1(t)&=q^N_1(0)+\frac{1}{N}\left(\mathcal{M}_{  A} (t) - \mathcal{M}_{ 1, D}(t)\right) + \int_0^t (1-I_0^N(s))\lambda(s)\dif s-\int_0^t(q^N_1(s)-q^N_2(s))\dif s,\\
q^N_i(t)&=q^N_i(0)+\frac{1}{N}\left(\mathcal{M}_{ A}(t)-\mathcal{M}_{ i, D}(t)\right)+\int_0^t I_0^N(s)\frac{q^N_{i-1}(s)-q^N_i(s)}{q^N_1(s)}\lambda(s)\dif s-\int_0^t (q^N_i(s)-q^N_{i+1}(s))\dif s, \quad i = 2,\dots,B,\\
\delta_0^N(t)&=\delta_0^N(0)+\frac{1}{N}\left(\mathcal{M}_{0}(t) - \mathcal{M}_{ A}(t) \right) +\mu\int_0^t u^N(s)\dif s -\int_0^tI_0^N(s)I_1^N(s)\lambda(s)\dif s,\\
\delta_1^N(t) &= \delta_1^N(0)+\frac{1}{N}\left(\mathcal{M}_{  A}(t)- \mathcal{M}_{1}(t)\right)+\int_0^t I_0^N(s)I_1^N(s)\lambda(s)\dif s-\nu\int_0^t\delta_1^N(s)\dif s.
\end{eq}
\end{table*}

\begin{table*}
\begin{eq}\label{eq:fluid mart3}
q^N_{1,j}(t)&=q^N_{1,j}(0)+\frac{1}{N}\mathcal{M}_{1,j} (t)  + \int_{[0,t]\times\mathcal{R}_1^c} r_j\lambda(s)\dif \alpha^N + \int_0^t\sum_{k=1}^K (q_{1,k}^N(s)-q_{2,k}^N(s))\gamma_kr_{k,j}\dif s\\
  &\hspace*{6cm}+ \int_0^t\sum_{k=1}^K (q_{2, k}^N(s) - q_{3 k}^N(s))\gamma_kr_{k,0}r_j\dif s - \gamma_j\int_0^t q_{1,j}^N(s)\dif s\\
q^N_{i,j}(t)&=q^N_{i,j}(0)+\frac{1}{N}\mathcal{M}_{ i,j}(t)+\int_{[0,t]\times\mathcal{R}_1} \frac{q^N_{i-1,j}(s)-q^N_{i,j}(s)}{\sum_{j=1}^Kq^N_{1,j}(s)}r_j\lambda(s)\dif \alpha^N+\int_0^t\sum_{k=1}^K (q_{ik}^N(s)-q_{i+1,k}^N(s))\gamma_kr_{k,j}\dif s\\
  &\hspace*{6cm}+ \int_0^t\sum_{k=1}^K (q_{i+1, k}^N(s) - q_{i+2, k}^N(s))\gamma_kr_{k,0}r_j\dif s - \gamma_j\int_0^t q_{i,j}^N(s)\dif s\\
\delta_0^N(t)&=\delta_0^N(0)+\frac{1}{N}\mathcal{M}_{0}(t) +\mu\int_0^t \bigg(1-\sum_{j=1}^Kq_{1,j}^N(s)-\delta_0^N(s)-\delta_1^N(s)\bigg)\dif s -\int_{[0,t]\times\mathcal{R}_2}\lambda(s)\dif s,\\
\delta_1^N(t) &= \delta_1^N(0)+\frac{1}{N}\mathcal{M}_{1}(t)+\int_{[0,t]\times\mathcal{R}_2}\lambda(s)\dif s-\nu\int_0^t\delta_1^N(s)\dif s.
\end{eq}
\end{table*}

\begin{proof}[Proof of Proposition~\ref{prop:mart zero1}]
 We only give proof for $\mathcal{M}_{ A}$ and the other cases can be proved similarly. Fix any $T>0$ and $\eta >0$.
 The proof makes use of the fact that the predictable quadratic variation process of a time-changed Poisson process is given by its compensator~\cite[Lemma 3.2]{PTRW07}.
  Using Doob's Martingale inequality~\cite[Theorem 1.9.1.3]{LS89}, we have
 \begin{align*}
  \Pro{\sup_{t\in [0,T]}\frac{\left| \mathcal{M}_{ A}(t)\right|}{N}>\varepsilon}&\leq \frac{1}{N^2\varepsilon^2}\E{\langle \mathcal{M}_{ A} \rangle_T}\\
  &\leq \frac{NT\sup_{t\in [0,T]}\lambda(t)}{N^2\varepsilon^2}\to 0,
 \end{align*}and the proof follows.
\end{proof}

\noindent
{\bf
Conditions of relative compactness.}
Let $(E,r)$ be a complete and separable metric space. For any $x\in D_E[0,\infty)$, $\kappa >0$ and $T>0$, define
\begin{equation}\label{eq:mod-continuity}
w'(x,\kappa,T)=\inf_{\{t_i\}}\max_i\sup_{s,t\in[t_{i-1},t_i)}r(x(s),x(t)),
\end{equation}
where $\{t_i\}$ ranges over all partitions of the form $0=t_0<t_1<\ldots<t_{n-1}<T\leq t_n$ with $\min_{1\leq i\leq n}(t_i-t_{i-1})>\kappa$ and $n\geq 1$.
 Below we state the conditions for the sake of completeness.
\begin{theorem}[{\cite[Corollary~3.7.4]{EK2009}}]\label{th:from EK}
Let $(E,r)$ be complete and separable, and let $\{X_n\}_{n\geq 1}$ be a family of processes with sample paths in $D_E[0,\infty)$. Then $\{X_n\}_{n\geq 1}$ is relatively compact if and only if the following two conditions hold:
\begin{enumerate}[{\normalfont (a)}]
\item For every $\eta>0$ and rational $t\geq 0$, there exists a compact set $\Gamma_{\eta, t}\subset E$ such that $$\varliminf_{n\to\infty}\Pro{X_n(t)\in\Gamma_{\eta, t}}\geq 1-\eta.$$
\item For every $\eta>0$ and $T>0$, there exists $\kappa>0$ such that
$$\varlimsup_{n\to\infty}\Pro{w'(X_n,\kappa, T)\geq\eta}\leq\eta.$$
\end{enumerate}
\end{theorem}

\begin{proof}[Proof of Lemma~\ref{lem:rel compactness}]
Note from \cite[Proposition 3.2.4]{EK2009} that, to prove the relative compactness of $(\vv^N(\cdot),\alpha^N)$, it is enough to prove relative compactness of the individual components.

Let $\mathfrak{L}_t$ denote the collection of measures $\gamma^t$ where $\gamma^t$ is the restriction of $\gamma$ on $[0,t]\times\bar{\mathbbm{Z}}_+^2$. Note that, by Prohorov's theorem,  $\mathfrak{L}_t$ is compact, since $\bar{\mathbbm{Z}}_+^2$ is compact. The topology on $\mathfrak{L}$ is defined such that any sequence $\{\gamma_N\}_{N\geq 1}$ is relatively compact in $\mathfrak{L}$ if and only if $\{\gamma_N^t\}_{N\geq 1}$ is relatively compact in $\mathfrak{L}_t$ for any $t>0$. Since $\mathfrak{L}_t$ is compact, any sequence $\{\gamma_N\}_{N\geq 1}$ is relatively compact in $\mathfrak{L}$.  Thus, the relative compactness of $\alpha^N$ follows.
To see the relative compactness of $\{\mathbf{v}^N(\cdot)\}_{n\geq 1}$, first observe that $E$ is compact and hence the compact containment condition (a) of Theorem~\ref{th:from EK} is satisfied trivially by taking $\Gamma_{\eta,t}\equiv E$. 

Let $\{\mathbf{M}^N(t)\}_{t\geq 0}$ denote the vector of all the martingale quantities appearing in \eqref{eq:fluid mart2}. Denote by $\|\cdot\|$, the Euclidean norm. For condition (b), we can see that, for any $0\leq t_1<t_2<\infty$,
\begin{equation}\label{mart-norm-ub}
\|\mathbf{v}^N(t_1)-\mathbf{v}^N(t_2)\|\leq C (t_2-t_1)+\frac{1}{N}\|\mathbf{M}^N(t_1)-\mathbf{M}^N(t_2)\|,
\end{equation}
for a sufficiently large constant $C>0$ where we have used $q_i^N\leq 1$, for all $i$, $\lambda(t)$ is bounded, and the fact that $(q_{i-1}^N-q_i^N)/q_1^N \leq 1$.
From Proposition~\ref{prop:mart zero1}, we get, for any $T\geq 0$,
$$\sup_{t\in[0,T]}\frac{1}{N}\|\mathbf{M}^N(t)\|\pto 0.$$
Now, the proof of the relative compactness of $(\mathbf{v}^N(t))_{t\geq 0}$ is complete if we can show that for any $\eta>0$, there exists a $\delta >0$ and a partition $(t_i)_{i\geq 1}$ with $\min_i|t_{i}-t_{i-1}|>\delta$ such that 
\begin{equation}
\varlimsup_{N\to\infty}\Pro{\max_i \sup_{s,t\in [t_{i-1},t_i)}\|\mathbf{v}^N(s)-\mathbf{v}^N(t)\| \geq \eta } < \eta.
\end{equation}Now, \eqref{mart-norm-ub} implies that, for any partition $(t_i)_{i\geq 1}$,
\begin{align*}
\max_i \sup_{s,t\in [t_{i-1},t_i)} \|\mathbf{v}^N(s)-\mathbf{v}^N(t)\|&\leq C \max_i (t_{i}-t_{i-1})+\zeta_N,
\end{align*}where $\Pro{\zeta_N>\eta/2}<\eta$ for all sufficiently large $N$. Now take $\delta = \eta/4C$ and any partition with $\max_i(t_i-t_{i-1})< \eta/2C$ and $\min_i(t_i-t_{i-1})>\delta$. Now on the event $\{\zeta_N\leq \eta/2\}$,  $$\max_i \sup_{s,t\in [t_{i-1},t_i)}\|\mathbf{v}^N(s)-\mathbf{v}^N(t)\| \leq \eta.$$
Therefore, for all sufficiently large $N$,
\begin{eq}
&\Pro{\max_i \sup_{s,t\in [t_{i-1},t_i)}\|\mathbf{v}^N(s)-\mathbf{v}^N(t)\| \geq \eta }\\&\hspace{3cm}
\leq \Pro{\zeta_N>\eta/2}\leq \eta,
\end{eq}and the proof of the relative compactness of $(\mathbf{v}^N(t))_{t\geq 0}$ is now complete.
The fact that the limit $(\vv,\alpha)$ of any convergent subsequence of $(\vv^N,\alpha^N)$ satisfies~\eqref{eq:rel compact}, follows by applying the continuous-mapping theorem.
\end{proof}

\begin{proof}[Proof of Theorem~\ref{th: fluid gen service}]The proof of Theorem~\ref{th: fluid gen service} is identical to the proof of Theorem~\ref{th: fluid}, which starts again by  establishing the martingale decomposition for $q_{ij}^N$ of the form \eqref{eq:fluid mart3}.
The definitions of the sets $\mathcal{R}_1$, $\mathcal{R}_2$ remain exactly the same. Thus the convergence result Lemma~\ref{lem:rel compactness} holds for $\qq^N=(q_{ij}^N)_{1\leq i\leq B, 1\leq j\leq K}$. The arguments for the time scale separation part remain unchanged as well, except the transition rate $(Z_1,Z_2)\rightarrow (Z_1,Z_2)+(0,1)$ in \eqref{eq:stationary-bd-chain} changes to $\sum_{j=1}^K(q_{1j}-q_{2j})+\nu\delta_1$.
\end{proof}

\section{Convergence of stationary distribution}\label{app:globstab}
\begin{proof}[Proof of Proposition~\ref{prop:glob-stab}]
The proof follows in three steps: in Lemma~\ref{lem:q1}, we show that $q_1(t)\to \lambda$ as $t\to\infty$, using this we show in Lemma~\ref{lem:q2} that $q_2(t)\to 0$, and then finally we deduce that $\delta_0(t)\to 1-\lambda$ and $\delta_1(t)\to 0$.

\begin{lemma}\label{lem:q1}
$q_1(t)\to \lambda$ as $t\to\infty$.
\end{lemma}
\begin{proof}
We first state four useful basic facts based on the fluid limit in Theorem~\ref{th: fluid}.
These are then used to prove Claims~\ref{claim:inf} and~\ref{claim:sup} which together imply Lemma~\ref{lem:q1}.
\begin{fact} \label{fact:nondec}
$q_1(t)$ is nondecreasing if $q_1(t)-q_2(t)\leq\lambda$. 
In particular, if $q_1(t)\leq\lambda$, then $q_1(t)$ is nondecreasing.
\end{fact}
\begin{claimproof}
Note that the rate of change of $q_1(t)$ is determined by $\lambda p_0(\qq(t), \dd(t))-q_1(t) +q_2(t)$. 
So it suffices to show that the latter quantity is non-negative when $q_1(t) - q_2(t)\leq \lambda$.
This follows directly from the fact that 
\begin{equation}\label{eq:lb-p0}
p_0(\qq(t), \dd(t))\geq \min\big\{\lambda^{-1}(\delta_1(t)\nu+q_1(t)-q_2(t)),1\big\}.
\end{equation}
\end{claimproof}
\noindent Define the subset $\mathcal{X}\subseteq E$ as
$$\mathcal{X}:= \Big\{(\qq,\dd)\in E: q_1 + \delta_0 + \delta_1 = 1,  \delta_1\nu+q_1-q_2\leq \lambda\Big\},$$
and denote by  $\indn{\cX}(\qq(s),\dd(s))$  the indicator of the event that $(\qq(s),\dd(s))\in \cX.$
Observe that $q_1(t)$ can be written as
\begin{equation}
\begin{split}
q_1(t)&= q_1(u) +\int_{u}^t\delta_1(s)\nu\indn{\cX}(\qq(s),\dd(s))\dif s\\
& +\int_{u}^t[\lambda - q_1(s) + q_2(s)] \indn{\cX^c}(\qq(s),\dd(s))\dif s.
\end{split}
\end{equation}
The above representation leads to Facts~\ref{fact:2},~\ref{fact:3} stated below.
\begin{fact}\label{fact:2}
\begin{align*}
q_1(t)\geq q_1(u) +\int_{u}^t[\lambda - q_1(s) + q_2(s)] \indn{\cX^c}(\qq(s),\dd(s))\dif s.
\end{align*}
\end{fact}
\begin{fact}\label{fact:3}
\begin{align*}
q_1(t)\geq q_1(u)+\nu\int_{u}^t\delta_1(s)\dif s-(\nu+1)\int_{u}^t\indn{\cX^c}(\qq(s),\dd(s))\dif s.
\end{align*}
\end{fact}

\begin{fact}\label{fact:4}
For all sufficiently small $\varepsilon>0$,
\begin{align*}
\xi(t) &\geq  \int_0^t\Big(\lambda-\frac{\varepsilon \nu}{2}-q_1(s)\Big)\dif s 
- \int_0^t\ind{u(s)>0}\dif s\\
&\hspace{3cm} - \int_0^t\ind{\delta_1(s)>\varepsilon/2}\dif s.
\end{align*}
\end{fact}
\begin{claimproof}
Observe that
\begin{align*}
\xi(t)&= \int_0^t\lambda(1-p_0(\qq(s),\dd(s),\lambda))\ind{\delta_0(s)>0}\dif s\\
&\geq \int_{0}^t\lambda(1-p_0(\qq(s),\dd(s),\lambda))\ind{\delta_0(s)>0, u(s) = 0,\delta_1(s)\leq \varepsilon/2}\dif s,
\end{align*}
and on the set $\{s:\delta_0(s)>0, u(s) = 0,\delta_1(s)\leq \varepsilon/2\}$ we have  $p_0(\qq(s),\dd(s),\lambda)\leq \lambda^{-1}(\varepsilon\nu/2+q_1(s))$. Therefore,
\begin{align*}
\xi(t)
&\geq \int_{0}^t\Big(\lambda-\frac{\varepsilon\nu}{2}-q_1(s)\Big)\ind{\delta_0(s)>0, u(s) = 0,\delta_1(s)\leq \varepsilon/2}\dif s.
\end{align*}
Moreover, if $\delta_0(s)=0, u(s) = 0,\delta_1(s)\leq \varepsilon/2,$ then $q_1(s)\geq 1-\varepsilon/2$, and for $\varepsilon<2(1-\lambda)/[1-\nu]^+$ we have $\lambda-\varepsilon\nu/2-q_1(s)<0$. Thus we finally obtain that
\begin{align*}
\xi(t)&\geq \int_{0}^t\Big(\lambda-\frac{\varepsilon \nu}{2}-q_1(s)\Big)\ind{ u(s) = 0,\delta_1(s)\leq \varepsilon/2}\dif s \\
&\geq \int_0^t\Big(\lambda-\frac{\varepsilon\nu}{2}-q_1(s)\Big)\dif s - \int_0^t\ind{u(s)>0}\dif s \\
&\hspace{3cm}- \int_0^t\ind{\delta_1(s)>\varepsilon/2}\dif s,
\end{align*}
where the second inequality follows from $\lambda-\varepsilon\nu/2-q_1(s)\leq \lambda <1$.
\end{claimproof}
\noindent
In order to break down the proof of Lemma~\ref{lem:q1}, we will establish the following two claims.
\begin{claim}\label{claim:inf}
$\liminf_{t\to\infty}q_1(t)\geq \lambda$.
\end{claim}
\begin{claimproof}
Assume the contrary. 
Using Fact~\ref{fact:nondec}, $q_1(t)$ is non-decreasing when $q_1(t)\leq \lambda$, and thus there must exist an $\varepsilon>0$, such that 
\begin{equation}\label{eq:contradict2}
q_1(t)\leq \lambda-\varepsilon\nu, \quad \forall\ t\geq 0.
\end{equation}
By Fact~\ref{fact:2} there exist positive constants $K_1,K_2$ (possibly depending on $\varepsilon$) such that $\forall\ t\geq 0$
\begin{equation}\label{eq:K1-choice}
\int_0^t \indn{\cX^c}(\qq(s),\dd(s))\dif s<K_1
\implies \int_0^t\ind{u(s)>0}\dif s<K_1,
\end{equation}
and by Fact~\ref{fact:3}, and \eqref{eq:K1-choice}
\begin{equation}\label{eq:K2-choice}
 \int_0^t\delta_1(s)\dif s<K_1 \implies \int_0^t\ind{\delta_1(s)>\frac{\varepsilon}{2}}\dif s<K_2.
\end{equation}
Note that since $\delta_1(t) = \delta_1(0) + \xi(t) - \nu\int_0^t\delta_1(s)\dif s,$
it must be the case that $\limsup_{t\to\infty}\xi(t)<\infty.$
On the other hand, Fact~\ref{fact:4}, together with \eqref{eq:K1-choice},~and~\eqref{eq:K2-choice}, implies that $\xi(t)\to\infty$ as $t\to\infty$, which leads to a contradiction.
\end{claimproof}
\begin{claim}\label{claim:sup}
$\limsup_{t\to\infty} q_1(t)\leq \lambda$.
\end{claim}
\begin{claimproof}
Suppose not, i.e., $\limsup_{t\to\infty}q_1(t) = \lambda+\varepsilon$ for some $\varepsilon>0$.
Because $q_1(t)$ is non-decreasing by Fact~\ref{fact:nondec} when $q_1(t)\leq \lambda$, there must exist a $t_0$ such that $q_1(t)\geq \lambda$ $\forall\ t\geq t_0$.
In that case,
\begin{align*}
&\sum_{i=1}^B q_i(t) \\
=& \sum_{i=1}^Bq_i(t_0) + \lambda\int_{t_0}^t\sum_{i=1}^B p_{i-1}(\qq(s),\dd(s),\lambda)\dif s - \int_{t_0}^t q_1(s)\dif s\\
 \leq& \sum_{i=1}^Bq_i(t_0)  - \int_{t_0}^t [q_1(s) - \lambda]^+\dif s,
\end{align*}
and thus,
$$\int_{t_0}^t [q_1(s) - \lambda]^+\dif s\leq \sum_{i=1}^B q_i(t) - \sum_{i=1}^Bq_i(t_0)<\infty.$$
This provides a contradiction with $\limsup_{t\to\infty} q_1(t) = \lambda +\varepsilon$, since the rate of decrease of $q_1(t)$ is at most~1.
\end{claimproof}
\end{proof}

\begin{lemma}\label{lem:q2}
$q_2(t)\to 0$ as $t\to\infty$.
\end{lemma}
\begin{proof}
Lemma~\ref{lem:q1} implies that for any $M,\varepsilon>0$, there exists $T(\varepsilon,M)<\infty$, such that $|q_1(t) - \lambda|\leq \varepsilon/M$ for all $t\geq T(\varepsilon,M).$
We will show that $\limsup_{t\to\infty} q_2(t) = 0$. 
Suppose not, i.e., $q_2(T)>\varepsilon>0$ for some $T>T(\varepsilon,M).$
Since the rate of decrease of $q_2(t)$ is at most $q_2(t)$, it follows that $q_2(t)\geq 9\varepsilon/16$ for all $t\in [T,T+1/2]$, and hence 
\begin{equation}\label{eq:diff-l-q1-q2}
q_1(t)-q_2(t)\leq \lambda + \frac{\varepsilon}{M}-\frac{9\varepsilon}{16}\leq \lambda -\frac{\varepsilon}{2},
\end{equation}
for $M\geq 16$.
Due to Fact~\ref{fact:2},
$$q_1\Big(T+\frac{1}{2}\Big)-q_1(T)\geq \frac{\varepsilon}{2}\int_T^{T+\frac{1}{2}}\indn{\cX^c}(\qq(s),\dd(s))\dif s.$$
Since $$q_1\Big(T+\frac{1}{2}\Big)-q_1(T)\leq \frac{2\varepsilon}{M},$$
it follows that
\begin{equation}\label{eq:indicator-upper2}
\int_T^{T+\frac{1}{2}}\indn{\cX^c}(\qq(s),\dd(s))\dif s\leq \frac{4}{M}.
\end{equation} 
Also, Fact~\ref{fact:3} yields
$$q_1\Big(T+\frac{1}{2}\Big)-q_1(T)\geq \nu\int_T^{T+\frac{1}{2}}\delta_1(s)\dif s - \frac{4\nu(\nu+1)}{M}.$$
Again since $$q_1\Big(T+\frac{1}{2}\Big)-q_1(T)\leq \frac{2\varepsilon}{M},$$
it follows that
\begin{equation}\label{eq:delta-upper}
\nu\int_T^{T+\frac{1}{2}}\delta_1(s)\dif s\leq \frac{4\nu(\nu+1)+2\varepsilon}{M}\leq \frac{5\nu(\nu+1)}{M},
\end{equation}
for $\varepsilon$ sufficiently smaller than $\nu$.
We will now proceed to show that~\eqref{eq:delta-upper} yields a contradiction.
Notice that
\begin{align*}
\delta_1(t) &= \delta_1(T) + \int_T^t\lambda(1-p_0(\qq(s),\dd(s),\lambda))\ind{\delta_0(s)>0}\dif s \\
&\hspace{3cm}- \nu \int_T^t\delta_1(s)\dif s\\
&\geq  \int_T^t (\lambda-q_1(s)+q_2(s))\indn{\cX^c}(\qq(s),\dd(s))\dif s \ind{\delta_0(s)>0}\dif s\\
&\hspace{4cm}-2\nu\int_T^t\delta_1(s)\dif s.
\end{align*}
Using \eqref{eq:diff-l-q1-q2}, we obtain for all $t\in [T,T+1/2]$,
\begin{align*}
&\delta_1(t)\geq -2\nu\int_T^t\delta_1(s)\dif s + \frac{\varepsilon}{2}\int_T^t \indn{\cX}(\qq(s),\dd(s))\ind{\delta_0(s)>0}\dif s\\
&\hspace{.2cm}\geq -2\nu\int_T^t\delta_1(s)\dif s + (t-T)\frac{\varepsilon}{2} -\frac{\varepsilon}{2}\int_T^t\indn{\cX^c}(\qq(s),\dd(s))\dif s\\
&\hspace{4cm}-\frac{\varepsilon}{2}\int_T^t\ind{u(s)=0,\delta_0(s)=0}\dif s\\
&\hspace{.2cm}\geq -2\nu\int_T^{T+\frac{1}{2}}\delta_1(s)\dif s-\frac{\varepsilon}{2}\int_T^{T+\frac{1}{2}}\indn{\cX^c}(\qq(s),\dd(s))\dif s  \\
&\hspace{2cm}+(t-T)\frac{\varepsilon}{2}-\frac{\varepsilon}{2}\int_T^{T+\frac{1}{2}}\ind{u(s)=0,\delta_0(s)=0}\dif s,
\end{align*}
and using~\eqref{eq:indicator-upper2} and~\eqref{eq:delta-upper}, it follows that
\begin{align*}
\delta_1(t)&\geq -\frac{10\nu(\nu+1)}{M}-\frac{2\varepsilon}{M}+(t-T)\frac{\varepsilon}{2}\\
&\hspace{2.5cm}-\frac{\varepsilon}{2}\int_T^{T+\frac{1}{2}}\ind{\delta_1(s)\geq (1-\lambda-\varepsilon/M)}\dif s\\
&\geq \frac{\varepsilon}{16}\quad\mbox{for all}\quad t\in \Big[T+\frac{1}{4}, T+\frac{1}{2}\Big],
\end{align*}
for $M$ sufficiently large, and observing that due to~\eqref{eq:delta-upper}, 
$$\int_T^{T+\frac{1}{2}}\ind{\delta_1(s)\geq (1-\lambda-\varepsilon/M)}\dif s \leq \frac{5(\nu+1)}{M(1-\lambda-\frac{\varepsilon}{M})}\leq \frac{10(\nu+1)}{M(1-\lambda)},$$
for $\varepsilon$ small enough.
\end{proof}

Since $q_1(t) - q_2(t)\to \lambda$ and $q_2(t)\to 0$, as $t\to\infty$,
it follows from \eqref{eq:lb-p0} that $p_0(\mathbf{q}(t),\boldsymbol{\delta}(t),\lambda)\to 1$ as $t\to\infty.$
Also, an application of Gronwall's inequality to
$$\delta_1(t) = \delta_1(0)+\int_o^t\lambda(1-p_0(\mathbf{q}(s),\boldsymbol{\delta}(s),\lambda))\dif s- \int_0^t\delta_1(s)\nu\dif s,$$
yields $\delta_1(t)\to 0$ as $t\to\infty$.
Consequently, $\delta_0(t)\to 1-\lambda$ as $t\to\infty$. This completes the proof of Proposition~\ref{prop:glob-stab}.
\end{proof}

\begin{proof}[Proof of Proposition~\ref{thm:limit interchange}]
Note that the proof of the proposition follows from~\cite[Corollary 2]{BenB08}. The arguments are sketched briefly for completeness.

Observe that $\pi^{N}$ is defined on $E$, and $E$ is a compact set. 
Prohorov's theorem implies that $\pi^{N}$ is relatively compact, and hence, has a convergent subsequence. Let $\{\pi^{N_n}\}_{n\geq 1}$ be a convergent subsequence, with $\{N_n\}_{n\geq 1}\subseteq\N$, such that $\pi^{N_n}\dto\hat{\pi}$ as $n\to\infty$. 
We will show that $\hat{\pi}$ is unique and equals the measure $\pi.$

Notice that if $(\qq^{N_n}(0),\dd^{N_n}(0))\sim\pi^{N_n}$, then  we know $(\qq^{N_n}(t),\dd^{N_n}(t))\sim\pi^{N_n}$ for all $t\geq 0$. 
Also, the process $(\qq^{N_n}(t),\dd^{N_n}(t))_{t\geq 0}$ converges weakly to $\{(\qq(t),\dd(t))\}_{t\geq 0}$, and $\pi^{N_n}\dto\hat{\pi}$ as $n\to\infty$. 
 Thus, $\hat{\pi}$ is an invariant distribution of the deterministic process $\{(\qq(t),\dd(t))\}_{t\geq 0}$.
 This in conjunction with the global stability in Proposition~\ref{prop:glob-stab} implies that $\hat{\pi}$ must be the fixed point of the fluid limit. 
 Since the latter fixed point is unique, we have shown the convergence of the stationary measure.
\end{proof}
%

\end{document}